\documentclass[a4paper,11pt]{article}
\usepackage{amsmath,amssymb,amsthm}
\usepackage{cite}
\usepackage{eucal}
\usepackage{graphicx}
\usepackage{microtype}
\usepackage[colorlinks=true,linkcolor=blue,citecolor=blue,urlcolor=blue]{hyperref}
\usepackage{iftex}
\ifPDFTeX
  \input{glyphtounicode}
\fi
\allowdisplaybreaks

\hypersetup{
  pdftitle={Contracting Transport Maps on Riemannian Manifolds},
  pdfauthor={Shrey Aryan, Bang-Xian Han and Zhuo-Nan Zhu},
  pdfsubject={Measure-preserving contractions on Riemannian manifolds via inverse mean curvature flow},
  pdfkeywords={measure-preserving contractions,
    inverse mean curvature flow, convex hypersurfaces,
    free-boundary hypersurfaces, Alexandrov surfaces,
    weighted spherical contractions, conformal doubling,
    spectral comparison}
}

\numberwithin{equation}{section}
\theoremstyle{plain}
\newtheorem{theorem}{Theorem}[section]
\newtheorem{corollary}[theorem]{Corollary}
\newtheorem{proposition}[theorem]{Proposition}
\newtheorem{lemma}[theorem]{Lemma}
\newtheorem*{conjecture}{Conjecture}
\theoremstyle{remark}
\newtheorem{remark}[theorem]{Remark}

\newcommand{\B}{\mathbb B}
\newcommand{\Sph}{\mathbb S}
\newcommand{\Lip}{\operatorname{Lip}}
\newcommand{\vol}{\operatorname{vol}}
\newcommand{\Hh}{\mathcal H}
\newcommand{\R}{\mathbb R}
\newcommand{\norm}[1]{\left\lVert #1\right\rVert}
\newcommand{\dist}{\mathsf d}
\renewcommand{\d}{\,\mathrm d}

\title{ \bf Contracting Transport Maps on Riemannian Manifolds}
\author{
Shrey Aryan\thanks{Department of Mathematics, Massachusetts Institute of
Technology, 77 Massachusetts Avenue, Cambridge, MA 02139-4307, USA.
Email: shrey183@mit.edu}\and
Bang-Xian Han\thanks{School of Mathematics, Shandong University,
Jinan 250100, China. Email: hanbx@sdu.edu.cn}
\and Zhuo-Nan Zhu\thanks{School of Mathematical Sciences, University of
Science and Technology of China, Hefei 230026, China. Email:
zhuonanzhu@mail.ustc.edu.cn}}
\date{}

\begin{document}

\maketitle
\begin{abstract} 
We use inverse mean curvature flow to construct bi-Lipschitz maps that preserve normalized volume and decrease distances.  These maps send a round sphere onto any smooth closed strictly convex hypersurface in a sphere and a flat disk onto any smooth strictly convex free-boundary disk in a Euclidean ball.

In dimension two, this proves a conjecture of E.~Milman for every smooth two-sphere with Gaussian curvature at least one and gives an analogous result for nonnegatively curved disks whose boundary has geodesic curvature one.  The spherical result proves the two-dimensional case of the spectral comparison conjectured by Colding and Minicozzi.  Counterexamples in dimensions $n\geq3$ show that the restriction to dimension two is sharp.

Furthermore, an equivariant extension of this construction yields, for every $n\geq2$, a contracting transport map from the uniform probability measure on a round hemisphere to the uniform probability measure on any closed geodesically convex subset of positive volume. This settles the remaining uniform-target case of a question raised by Beck and Jerison. In dimension two, we also find geometric conditions under which the uniform measure on the hemisphere can be transported by a contracting map to a broad class of nonuniform probability measures supported on domains in a hemisphere.
\end{abstract}



\setcounter{tocdepth}{1}
\begingroup
\setlength{\parskip}{0pt}
\tableofcontents
\endgroup

\clearpage
\section{Introduction}
\subsection*{Motivation and background}
Caffarelli's contraction theorem states that the transport map minimizing the quadratic distance cost from the standard Gaussian measure to a more log-concave probability measure is $1$-Lipschitz \cite{Caffarelli}.  Consequently, many functional and geometric inequalities for the Gaussian measure transfer to more log-concave measures.  A similar principle appears in comparison geometry: functional and geometric inequalities on model spaces, such as the round sphere, extend to manifolds satisfying appropriate curvature lower bounds.  The standard Gaussian space is a model for the curvature-dimension condition $\mathrm{CD}(1,\infty)$, just as the round sphere $\Sph^n$ is a model for
Riemannian manifolds satisfying $\operatorname{Ric}\geq(n-1)g$. Motivated by this analogy, E.~Milman proposed the following conjecture \cite[Conjecture~4]{Milman}.
\begin{conjecture}[E. Milman]
Let $(M^n,g)$ be a smooth Riemannian manifold diffeomorphic to $\Sph^n$ and
suppose that $\operatorname{Ric}_g\geq(n-1)g$.  Then there exists a map
\[
 T:(\Sph^n,g_{\mathrm{can}})\longrightarrow(M,g)
\]
such that
\[
 T_\#\frac{\d\vol_{\mathrm{can}}}{|\Sph^n|}
 =\frac{\d\vol_g}{\vol_g(M)},
 \qquad \Lip(T)\leq1.
\]
\end{conjecture}
A single measure-preserving contraction implies several classical comparison
results.  To see this, write
$\sigma_n:=\d\vol_{\mathrm{can}}/|\Sph^n|$ and
$\mu_g:=\d\vol_g/\vol_g(M)$.
If $T_\#\sigma_n=\mu_g$ and $\Lip(T)\leq1$, then the normalized round sphere
dominates $(M,g,\mu_g)$ in Gromov's metric--measure Lipschitz order
\cite{Gromov}.  Moreover, surjectivity and the contraction property give
\[
 \operatorname{diam}(M,g)\leq\pi,
 \qquad
 \mu_g\bigl(B_r(Tx)\bigr)\geq\sigma_n\bigl(B_r(x)\bigr),
 \qquad x\in\Sph^n, r>0.
\]
These are the Bonnet--Myers diameter bound and the normalized Bishop volume
comparison.  For every Borel set $A\subseteq M$,
\[
 \bigl(T^{-1}(A)\bigr)_r\subseteq T^{-1}(A_r),
\]
where the two $r$-neighborhoods are taken in the source and target,
respectively.  Hence concentration inequalities and, through Minkowski
boundary measures, the L\'evy--Gromov isoperimetric comparison pass from the
round sphere to $M$.  Pullback of functions gives Poincar\'e and
log-Sobolev inequalities and, by the min--max principle, spectral estimates
such as the Lichnerowicz bound $\lambda_1(M,g)\geq n$.  Thus the map in
Milman's conjecture implies the Bonnet--Myers, normalized Bishop,
L\'evy--Gromov, and Lichnerowicz comparisons simultaneously; see
\cite{Milman,Serres}.

Previous affirmative results on E. Milman's conjecture were perturbative.
Serres proved the conjecture in dimension two for metrics sufficiently close to a
round metric of radius $\rho<1$, using Ricci flow and the Kim--Milman
construction \cite{KimMilman,Serres}; Ge and Serres extended this perturbative
result to arbitrary dimensions using the Brenier--McCann map \cite{GeSerres}.
The allowed perturbation size degenerates as $\rho\to 1$.

The spectral consequences of a contraction give an obstruction in higher
dimensions.  The first author constructed counterexamples in
dimensions $n\geq4$ \cite{Aryan}, and Lin, Wang, and Xu did so in dimension
three \cite[Section~16]{LinWangXu}.
Lin, Wang, and Xu proved the corresponding spectral comparison and rigidity
theorem for smooth metrics on $\Sph^2$ with Gaussian curvature at least one
\cite[Theorem~1.8]{LinWangXu}, while the existence of a contraction preserving
normalized volume remained open.

\subsection*{Main results} 

Among dimensions $n\geq2$, the higher-dimensional counterexamples leave
dimension two as the only one in which E. Milman's conjecture can hold in full
generality.  In this dimension, $K_g\geq1$ is precisely the Ricci lower bound
in the conjecture, and the spherical Weyl theorem makes the extrinsic flow
construction available.  This gives the following intrinsic theorem.
\begin{theorem}[E. Milman's conjecture for two-spheres]
\label{thm:milman}
Let $(M^2,g)$ be a smooth Riemannian manifold diffeomorphic to $\Sph^2$ and
assume $K_g\geq1$.  Then there is a bi-Lipschitz homeomorphism
\[
 T:(\Sph^2,g_{\mathrm{can}})\longrightarrow(M,g)
\]
such that
\[
 T_\#\frac{\d\vol_{\mathrm{can}}}{|\Sph^2|}
 =\frac{\d\vol_g}{\vol_g(M)},\qquad L:=\Lip(T)\leq1.
\]
With $a=\vol_g(M)/|\Sph^2|$,
\begin{equation}\label{eq:intro-milman-bilip}
 \frac{a}{L}\,\dist_{\mathrm{can}}(z,z')
 \leq \dist_g(Tz,Tz')
 \leq L\dist_{\mathrm{can}}(z,z'),
 \qquad \forall\,z,z'\in\Sph^2.
\end{equation}
If $K_g>1$, the map may be chosen with $L<1$.
\end{theorem}

The map in Theorem~\ref{thm:milman} yields the following quantitative
refinement of the spectral comparison of Lin, Wang, and Xu \cite{LinWangXu}.  If
$0=\lambda_0(M,g)\leq\lambda_1(M,g)\leq\cdots$, then
\begin{equation}
 L^{-2}\lambda_k(\Sph^2,g_{\mathrm{can}})
 \leq \lambda_k(M,g)
 \leq \left(\frac{L}{a}\right)^2
 \lambda_k(\Sph^2,g_{\mathrm{can}}),
 \qquad k\geq0.
\end{equation}

Since $L\leq1$, the lower estimate implies the usual comparison
$\lambda_k(M,g)\geq\lambda_k(\Sph^2,g_{\mathrm{can}})$ and is stronger when
$L<1$.  We also prove rigidity for the lower estimate: equality for some
$k\geq1$ holds if and only if $a=L^2$; in this case
$T:(\Sph^2,L^2g_{\mathrm{can}})\longrightarrow(M,g)$ is an isometry.  This
rigidity statement holds for every Lipschitz map from the
round sphere to a smooth Riemannian two-sphere that preserves normalized
volume.  In particular, Theorem~\ref{thm:spectral-rigidity} recovers the smooth
spectral rigidity theorem of Lin, Wang, and Xu
\cite[Theorem~1.8]{LinWangXu}.  It also settles the two-dimensional case of
the spectral comparison conjectured by Colding and Minicozzi in \cite{CM}.
The upper estimate and its equality case for bi-Lipschitz maps are discussed
in Remark~\ref{rem:spectral-upper-rigidity}.

The intrinsic theorem above is deduced from an all-dimensional extrinsic
result by inverse mean curvature flow.  Let $X_t$
solve $\partial_tX=H_t^{-1}\nu_t$, and let $g_t$ and $\d\vol_t$ denote the
pullbacks of the induced metric and volume measure to the parameter manifold.
Strict convexity and the first-variation formulas give
\begin{equation}
 \partial_t g_t=2H_t^{-1}h_t>0,
 \qquad
 \partial_t\d\vol_t=\d\vol_t.
\end{equation}
For $t>s$, the backward Lagrangian map
\[
 \Phi_{t,s}:=X_s\circ X_t^{-1}:\Sigma_t\longrightarrow\Sigma_s
\]
is therefore $1$-Lipschitz and preserves normalized volume.  In the closed
spherical setting the terminal model is an equator, whereas in the
free-boundary Euclidean setting it is a flat disk.  Passing the backward maps
to these terminal models yields the following two contractions.

\begin{theorem}[Extrinsic IMCF contractions]
\label{thm:extrinsic-unified}
Let $n\geq2$.
\begin{enumerate}
\item If $\Sigma^n\subseteq\Sph^{n+1}$ is a smooth, closed, connected,
embedded, strictly convex hypersurface, then there exists a bi-Lipschitz homeomorphism
\[
 T:(\Sph^n,g_{\mathrm{can}})\longrightarrow(\Sigma,g_\Sigma)
\]
such that
\[
 T_\#\frac{\d\vol_{\mathrm{can}}}{|\Sph^n|}
 =\frac{\d\vol_\Sigma}{|\Sigma|},
 \qquad L:=\Lip(T)<1.
\]
With
$b=|\Sigma|/|\Sph^n|$,
\begin{equation}\label{eq:intro-closed-bilip}
 \frac{b}{L^{n-1}}\dist_{\mathrm{can}}(z,z')
 \leq \dist_\Sigma(Tz,Tz')
 \leq L\dist_{\mathrm{can}}(z,z'),
 \qquad \forall\,z,z'\in\Sph^n.
\end{equation}
\item If $\Sigma^n\subseteq\overline{\B^{n+1}}$ is a smooth, compact,
connected, embedded, strictly convex hypersurface of disk type with free
boundary on $\Sph^n=\partial\overline{\B^{n+1}}$, then there exists a bi-Lipschitz
homeomorphism
\[
 T:(\overline{\B^n},g_{\mathrm{Euc}})\longrightarrow(\Sigma,g_\Sigma)
\]
such that
\[
 T_\#\frac{\d x}{\vol(\B^n)}
 =\frac{\d\vol_\Sigma}{|\Sigma|},
 \qquad L:=\Lip(T)<1.
\]
With
$b=|\Sigma|/\omega_n$, $\omega_n=\vol(\B^n)$,
\begin{equation}
 \frac{b}{L^{n-1}}|z-z'|
 \leq \dist_\Sigma(Tz,Tz')
 \leq L|z-z'|,
 \qquad \forall\,z,z'\in\overline{\B^n}.
\end{equation}
\end{enumerate}
\end{theorem}
The existence and terminal convergence of the relevant flows are due to
Makowski and Scheuer in the closed spherical case \cite{MakowskiScheuer}
and to Lambert and Scheuer in the free-boundary case
\cite{LambertScheuer}.  The endpoint
argument and the bi-Lipschitz lower bound are established in
Section~\ref{sec:mm-flow}.

The intrinsic results follow by combining
Theorem~\ref{thm:extrinsic-unified} with the Weyl embedding theorems.
Theorem~\ref{thm:milman} follows from the closed part of
Theorem~\ref{thm:extrinsic-unified} and Lu's spherical Weyl embedding theorem
\cite{Lu}.  The free-boundary part of
Theorem~\ref{thm:extrinsic-unified} has the following intrinsic counterpart,
based on Koerber's solution of the free-boundary Weyl problem \cite{Koerber}.

\begin{theorem}[Nonnegatively curved disks]\label{thm:disk}
Let $g$ be a smooth Riemannian metric on $\overline{\B^2}$.  Suppose
$K_g\geq0$ on $\overline{\B^2}$ and
$k_{\partial\overline{\B^2}}=1$,
where
$k_{\partial\overline{\B^2}}=-g(\nabla_\xi\xi,\eta)$, with $\xi$ a unit
tangent and $\eta$ the outward unit conormal.
Then there exists a bi-Lipschitz homeomorphism
\[
 T:(\overline{\B^2},g_{\mathrm{Euc}})
 \longrightarrow(\overline{\B^2},g)
\]
such that
\[
 T_\#\frac{\d x}{\omega_2}
 =\frac{\d\vol_g}{\vol_g(\B^2)},
 \qquad L:=\Lip(T)\leq1.
\]
With $a=\vol_g(\B^2)/\omega_2$,
\begin{equation}
 \frac{a}{L}|z-z'|\leq \dist_g(Tz,Tz')\leq L|z-z'|,
 \qquad \forall\,z,z'\in\overline{\B^2}.
\end{equation}
If $K_g>0$, the map may be chosen with $L<1$.
In particular,
\begin{equation}
 \lambda_k^{D/N}(\B^2,g)
 \geq L^{-2}\lambda_k^{D/N}(\B^2,g_{\mathrm{Euc}})
\end{equation}
for all $k\geq1$ in the Dirichlet case and all $k\geq0$ in the Neumann case.
\end{theorem}

The closed construction also yields contracting transport maps on a
hemisphere.  Beck and Jerison asked, among other things, whether the uniform
probability measure on a round hemisphere can be transported by an injective
$1$-Lipschitz map to a probability measure with log-concave density on a
geodesically convex subset \cite[Section~IV]{BeckJerison}.  Fathi, Fradelizi,
Gozlan and Zugmeyer showed that this fails in general, while observing that
their obstructions leave open the uniform-target subcase, in which the target
is the normalized uniform measure on a positive-volume geodesically convex
subset \cite[p.~159]{FathiFradeliziGozlanZugmeyer}.

Let $\Sph^n_+$ be a closed round hemisphere of
$(\Sph^n,g_{\mathrm{can}})$ and set
$\sigma_+:=\mathbf 1_{\Sph^n_+}\d\vol_{\mathrm{can}}/|\Sph^n_+|$.

A full-support target on the hemisphere must be uniform.  Indeed, suppose that
$T:\Sph^n_+\to\Sph^n_+$ is $1$-Lipschitz and
\[
 T_\#\sigma_+=\rho\,\d\vol_{\mathrm{can}},
 \qquad
 \operatorname{supp}(\rho\,\d\vol_{\mathrm{can}})=\Sph^n_+,
 \qquad \int_{\Sph^n_+}\rho\,\d\vol_{\mathrm{can}}=1.
\]
The support assumption and compactness of the image imply that $T$ is
surjective.  For every Borel set
$A\subseteq\Sph^n_+$, the area formula \cite{Kirchheim}, surjectivity, and
$J_T\leq1$ yield 
\[
 |\Sph^n_+|\int_A\rho\,\d\vol_{\mathrm{can}}
 =|T^{-1}(A)|
 \geq\int_{T^{-1}(A)}J_T\,\d\vol_{\mathrm{can}}
 \geq |A|.
\]
Hence $\rho\geq|\Sph^n_+|^{-1}$ almost everywhere, and normalization forces
equality.  Thus any nonuniform probability measure that occurs as such an image must be
supported on a proper subset of the hemisphere.

The uniform-target subcase left open in
\cite{FathiFradeliziGozlanZugmeyer} has an affirmative answer in every
dimension $n\geq2$.

\begin{theorem}[Convex subsets of a hemisphere]
	\label{thm:hemisphere-transport}
	Let $n\geq2$, and let $K\subseteq\Sph^n_+$ be closed and geodesically convex 
	with $|K|:=\int_K\d\vol_{\mathrm{can}}>0$. 
	We endow $K$ with the intrinsic length metric induced by
	$g_{\mathrm{can}}$.
	Since $K$ is geodesically convex, this metric agrees with the restriction
	of the ambient spherical distance.
	Then there exists a bi-Lipschitz homeomorphism
	\[
	T:(\Sph^n_+,g_{\mathrm{can}})
	\longrightarrow (K, g_{\mathrm{can}})
	\]
	such that
	\begin{equation*}
	 T_\#\sigma_+=\frac{\mathbf 1_K\d\vol_{\mathrm{can}}}{|K|}, 
		\qquad
		L:=\Lip(T)\leq1.
	\end{equation*}
	With $b_K:=|K|/|\Sph^n_+|$, 
	\begin{equation}\label{eq:hemisphere-contraction}
		\frac{b_K}{L^{n-1}}\dist_{\mathrm{can}}(z,z')
		\leq
		\dist_{\mathrm{can}}(Tz,Tz')
		\leq
		L\dist_{\mathrm{can}}(z,z'),
		\qquad
		\forall\,z,z'\in\Sph^n_+.
	\end{equation}
\end{theorem}

The uniform result follows by approximating $K$ with thin reflection-symmetric
strictly convex hypersurfaces and applying the equivariant closed construction.
In dimension two, an equivariant refinement of
Corollary~\ref{cor:alexandrov}, applied to metric doubles, yields the
corresponding contraction from the hemisphere onto an Alexandrov disk.

\begin{theorem}[Contractions onto Alexandrov disks]
\label{thm:alexandrov-disk}
Let $(Y,\dist_Y)$ be a compact Alexandrov surface of curvature at least one
and homeomorphic to a closed disk.  Then there exists a bi-Lipschitz
homeomorphism
\[
 T:(\Sph^2_+,g_{\mathrm{can}})\longrightarrow(Y,\dist_Y)
\]
such that
\[
 T_\#\sigma_+=\frac{\Hh_Y^2}{\Hh^2(Y)},
 \qquad L:=\Lip(T)\leq1.
\]
With $b_Y:=\Hh^2(Y)/|\Sph^2_+|$,
\begin{equation}
	 \frac{b_Y}{L}\dist_{\mathrm{can}}(z,z')
	\leq\dist_Y(Tz,Tz')
	\leq L\dist_{\mathrm{can}}(z,z'),
	\qquad \forall\,z,z'\in\Sph^2_+. 
\end{equation}
\end{theorem}

The weighted result follows by applying
Theorem~\ref{thm:alexandrov-disk} to a conformal metric.  Let
$\Omega\subseteq\Sph^2$ be a compact smooth topological disk, set
$g_0=g_{\mathrm{can}}|_\Omega$, and denote its intrinsic distance and volume
measure by $\dist_{0,\Omega}$ and $\d\vol_0$.  If $\eta$ is the outward unit
normal and $\tau$ a unit boundary tangent, set $\kappa_0=\langle\nabla^{g_0}_\tau\eta,\tau\rangle_{g_0}$. 
For $V$ normalized by $\max_{\overline\Omega}V=0$, consider the conformal
metric $\hat g=e^{-V}g_0$.  The assumptions below are precisely the
conditions that $\hat g$ have curvature at least one and convex boundary.
Thus $(\Omega,\hat g)$ is an Alexandrov disk of curvature at least one \cite[Section~2.9, Example~(1)]{BuragoGromovPerelman}. 

\begin{corollary}[Weighted contractions via a canonical conformal metric]
	\label{cor:weighted-hemisphere}
	Let $V\in C^2(\overline\Omega)$ be normalized by
	$\max_{\overline\Omega}V=0$ and assume
	\begin{equation}\label{eq:intro-weighted-conditions}
		e^V\left(1+\frac12\Delta_0V\right)\geq1
		\quad\text{on }\Omega,
		\qquad
		\partial_\eta V\leq2\kappa_0
		\quad\text{on }\partial\Omega.
	\end{equation}
	Set $Z_V:=\int_\Omega e^{-V}\d\vol_0$ and
	$m_V:=\min_{\overline\Omega}V$.
	Then there exists a bi-Lipschitz homeomorphism
	\[
	T:(\Sph^2_+,g_{\mathrm{can}})
	\longrightarrow(\Omega,g_0)
	\]
	such that
		\begin{equation*}
		T_\#\sigma_+
		=
		\frac{e^{-V}}{Z_V}\d\vol_0, 
	\end{equation*}
	and 
	\begin{equation}\label{eq:intro-weighted-bilip}
		e^{\frac{m_V}{2}}\frac{Z_V}{2\pi}\dist_{\mathrm{can}}(z,z')
		\leq
		\dist_{0,\Omega}(Tz,Tz')
		\leq
		\dist_{\mathrm{can}}(z,z'),
		\qquad
		\forall\,z,z'\in\Sph^2_+.
	\end{equation}
\end{corollary}

Although the density can be absorbed into the conformal metric
$e^{-2V/n}g_0$ in every dimension, the realization step used here is
presently two-dimensional: in higher dimensions the conformal double is an
arbitrary Alexandrov sphere, and no analogue of the Weyl realization and
contraction argument is available.

\subsection*{Further remarks and overview}

\paragraph{Free-boundary geometry.}
The boundary condition in Theorem~\ref{thm:disk} is the intrinsic form of the
free boundary condition.  For an isometric free-boundary embedding into
$\overline{\B^3}$, the position vector along the boundary is the outward unit
conormal.  Differentiating its orthogonality to a unit boundary tangent gives
$k_{\partial\overline{\B^2}}=1$ with our sign convention.  Conversely,
Koerber's theorem gives such an embedding when $K_g>0$.  Other applications
of inverse mean curvature flow with free boundary include monotonicity of the
first $p$-Laplacian eigenvalue \cite{HoPyo} and comparisons of volume and
torsional rigidity with the flat disk \cite{GimenoGonzalez}.

\paragraph{Regularity extensions.}
The smoothness assumption in Theorem~\ref{thm:extrinsic-unified} can be
weakened.  For every $0<\alpha<1$, its conclusions remain valid for
$C^{2,\alpha}$ initial hypersurfaces by short-time regularization; see
\cite[Theorem~2.5.7]{GerhardtCurvatureProblems} in the closed case and
\cite[Theorem~2.12]{MarquardtThesis} in the case with free boundary.
Together with \cite[Theorem~1.1]{Koerber}, this proof extends
Theorem~\ref{thm:disk} to
$g\in C^{k,\alpha}(\overline{\B^2})$, $k\geq4$.
Theorem~\ref{thm:milman} also has a counterpart for Alexandrov two-spheres;
see Corollary~\ref{cor:alexandrov}.

\paragraph{Strict spherical caps.}
Fix $N\in\Sph^2$, set $r=\dist_{\mathrm{can}}(N,\cdot)$, and let
$K_R:=\{r\leq R\}$, where $0<R<\pi/2$.
For the outward normal $\eta=\partial_r$, one has $\kappa_0=\cot R$.  Since
$K_R$ is geodesically convex, its intrinsic distance agrees with the ambient
spherical distance.  By Corollary~\ref{cor:weighted-hemisphere}, there exists a bi-Lipschitz contraction for the ambient spherical distance whenever
$V\in C^2(\overline{K_R})$ is normalized by $\max_{K_R}V=0$ and
\[
 e^V\left(1+\frac12\Delta_0V\right)\geq1
 \quad\text{in }K_R,
 \qquad
 \partial_rV\leq2\cot R
 \quad\text{on }\partial K_R.
\]
The radial potentials
\[
 V_a(r)=a(\cos R-\cos r),
 \qquad 0<a<2\cos R,
\]
satisfy
\[
 \nabla^2V_a=a\cos r\,g_0>0,\qquad
 \Delta_0V_a=2a\cos r,\qquad
 \partial_\eta V_a=a\sin R<2\cot R.
\]
They also satisfy the curvature inequality strictly.  Indeed, with
$s=\cos r$ and $c=\cos R$, the required inequality is equivalent to
$\log(1+as)>a(s-c)$; its left-hand side minus its right-hand side is decreasing
in $s\in[c,1]$, while
$\log(1+a)>a-a^2/2>a(1-c)$ because $a<2c$.
Consequently, $V_a$ and all sufficiently small nonradial $C^2$ perturbations
of $V_a$ form a $C^2$-open class of strictly log-concave target densities to
which Corollary~\ref{cor:weighted-hemisphere} applies.  The restriction
$R<\pi/2$ is essential.  When $R=\pi/2$, the conditions imply
$\Delta_0V\geq2(e^{-V}-1)\geq0$ and $\partial_\eta V\leq0$.  The divergence
theorem gives
\[
0\leq
\int_{\Sph^2_+}\Delta_0V\,\d\vol_0
=
\int_{\partial\Sph^2_+}\partial_\eta V\,\d s_0
\leq0.
\]
Hence
$e^{-V}-1=0$ and $V\equiv0$.  This is consistent with the fact that a probability measure with full support
on the whole hemisphere can be the image of $\sigma_+$ under a $1$-Lipschitz
map only when it is uniform.

\paragraph{Relation to quadratic optimal transport.}
The maps constructed here need not be optimal for the quadratic cost; see
\cite{Villani} for background on optimal transport.  For example, let
\[
 \Sph^2_+
 :=\left\{(x_1,x_2,x_3)\in\Sph^2:x_3\geq0\right\},
 \qquad
 K:=\left\{(x_1,x_2,x_3)\in\Sph^2:x_3\geq0,\ x_1\geq0\right\}.
\]
In geodesic polar coordinates $(\theta,\varphi)$ centered at
$e_1=(1,0,0)$, a quadratic-cost optimal map from $\sigma_+$ onto the normalized
uniform measure on $K$ is
\[
 T_{\mathrm{opt}}(\theta,\varphi)=(r(\theta),\varphi),
 \qquad
 \cos r(\theta)=\frac{1+\cos\theta}{2},
\]
which is well-defined almost everywhere.  The source and target have the same angular range, and
$1-\cos r(\theta)=(1-\cos\theta)/2$; comparison of the normalized area
elements $\sin\theta\,\d\theta\,\d\varphi$ and
$\sin r\,\d r\,\d\varphi$ therefore verifies the pushforward identity.
Moreover, the radial coordinate
$x\mapsto\dist_{\mathrm{can}}(e_1,x)$ is $1$-Lipschitz, so every transport plan
has cost at least the one-dimensional quadratic transport cost between the
radial marginals.  The displayed map couples the radial marginals by monotone rearrangement,
attains this lower bound, and is therefore optimal.  For the unit angular vector
$e_\varphi=(\sin\theta)^{-1}\partial_\varphi$, however,
\[
 \norm{\d T_{\mathrm{opt}}|_{(\theta,\varphi)}}_{\mathrm{op}}
 \geq
 \left\lvert
 \d T_{\mathrm{opt}}|_{(\theta,\varphi)}(e_\varphi)
 \right\rvert_{g_{\mathrm{can}}}
 =\frac{\sin r(\theta)}{\sin\theta}
 \to\infty
\]
as $\theta\to\pi$.  Thus the Brenier transport map is not Lipschitz, whereas
Theorem~\ref{thm:hemisphere-transport} supplies a measure-preserving
contraction for the same pair of measures.

\paragraph{Related work.}
Several works generalize Caffarelli's contraction theorem or provide new
proofs.  These include results in infinite-dimensional settings
\cite{feyel2004monge,mikulincer2024brownian}, for $1/d$-concave densities
\cite{carlier2024optimal}, and in curved spaces
\cite{fathi2024transportation,lopez2025bakry,GeSerres}.  In addition to the
work of Kim and Milman \cite{KimMilman}, recent proofs use entropic optimal
transport \cite{chewi2023entropic,fathi2020proof}.

\paragraph{Summary of the constructions.}
The main contraction constructions are summarized in the following table.
\begin{center}
\small
\renewcommand{\arraystretch}{1.25}
\begin{tabular}{@{}c c p{70mm}@{}}
\hline
\textbf{Result} & \textbf{Contraction} & \textbf{Construction}\\ \hline
Theorem~\ref{thm:extrinsic-unified}~(i)
 & $\Sph^n\to\Sigma^n$
 & closed spherical IMCF; terminal equator\\
Theorem~\ref{thm:extrinsic-unified}~(ii)
 & $\overline{\B^n}\to\Sigma^n$
 & free-boundary IMCF; terminal flat disk\\
Theorem~\ref{thm:hemisphere-transport}
 & $\Sph^n_+\to K$
 & thin reflection-symmetric hypersurfaces and an equivariant endpoint limit\\
Theorem~\ref{thm:alexandrov-disk}
 & $\Sph^2_+\to Y$
 & metric doubling and an equivariant Alexandrov contraction\\
Corollary~\ref{cor:weighted-hemisphere}
 & $\Sph^2_+\to\Omega$
 & canonical conformal metric and Theorem~\ref{thm:alexandrov-disk}\\
\hline
\end{tabular}
\end{center}

\paragraph{Organization.}
Section~\ref{sec:mm-flow} establishes the contraction principle and the inverse
distance bound.  Section~\ref{sec:closed} treats closed contractions: it first
constructs the endpoint map for strictly convex spherical hypersurfaces and
then proves the smooth and Alexandrov two-sphere results, including their
metric, spectral, and rigidity consequences.  Section~\ref{sec:free-boundary}
treats hypersurfaces with free boundary and proves the intrinsic disk theorem.
The uniform, Alexandrov-disk, and weighted constructions from a hemisphere
are proved in Section~\ref{sec:hemisphere}.

\section{Flow-generated metric--measure contractions}\label{sec:mm-flow}
All distances in this section are intrinsic length distances, including when
the underlying manifold has boundary.  For a finite positive measure $\mu$,
write $\bar\mu=\mu/\mu(M)$.

For compact metric--measure spaces with normalized measures, we use
Gromov's Lipschitz order
\cite[Chapter~$3\frac12$]{Gromov} and write
\[
 (Y,\dist_Y,\bar\nu)\preceq_{\mathrm{Lip}}
 (X,\dist_X,\bar\mu)
\]
if there exists a $1$-Lipschitz map $F:X\to Y$ with
$F_\#\bar\mu=\bar\nu$.  The relation is transitive by composition.  Under the
hypotheses below, for $\alpha(t)\geq0$, the backward maps show that
\[
 (M,g_s,\bar\mu_s)\preceq_{\mathrm{Lip}}
 (M,g_t,\bar\mu_t),\qquad s<t.
\]
Hence $t\mapsto(M,g_t,\bar\mu_t)$ is monotone in the Lipschitz order.
Proposition~\ref{prop:endpoint-contraction} includes a terminal model as the
maximal element of this curve.

\begin{proposition}[Metric--measure monotone flow]\label{prop:mm-flow}
Let $M$ be a compact connected smooth manifold, possibly with boundary.  Let
$g_t$, $t\in I$, be Riemannian metrics depending $C^1$-smoothly on $t$, and
let $\mu_t$ be smooth positive measures with the same time regularity.
Suppose there exist locally integrable functions
$\alpha,c:I\to\R$ such that
\begin{equation}\label{eq:mm-flow-hypotheses}
 \partial_t g_t\geq2\alpha(t)g_t,
 \qquad
 \partial_t\mu_t=c(t)\mu_t.
\end{equation}
Then for $s<t$ the identity
\[
 I_{t,s}:(M,g_t)\longrightarrow(M,g_s)
\]
preserves normalized measure and satisfies
\begin{equation}\label{eq:mm-flow-conclusion}
 \Lip(I_{t,s})\leq
 e^{-\int_s^t \alpha(r)\,\d r}.
\end{equation}
More generally, suppose $X:I\times M\to N$ is smooth and each
$X_t:=X(t,\cdot)$ is an embedding.  If $g_t=X_t^*\bar g$ and
$\mu_t=X_t^*\nu_t$ for measures $\nu_t$ on $X_t(M)$, then the transported map
\[
 \Phi_{t,s}:=X_s\circ X_t^{-1}:X_t(M)\longrightarrow X_s(M)
\]
has the same Lipschitz bound and satisfies
\[
 (\Phi_{t,s})_\#\frac{\nu_t}{\nu_t(X_t(M))}
 =\frac{\nu_s}{\nu_s(X_s(M))}.
\]
\end{proposition}

\begin{proof}
For each tangent vector $v$, apply Gr\"{o}nwall's inequality to
$r\mapsto g_r(v,v)$.  The first inequality in
\eqref{eq:mm-flow-hypotheses} then yields
$g_t\geq e^{2\int_s^t \alpha(r)\,\d r}g_s$.
Applying this quadratic-form inequality to the velocities of absolutely
continuous curves, and then taking infima of their lengths, proves
\eqref{eq:mm-flow-conclusion}.  Moreover, the measure equation in
\eqref{eq:mm-flow-hypotheses} yields
$\mu_t=e^{\int_s^t c(r)\,\d r}\mu_s$ and
$\mu_t(M)=e^{\int_s^t c(r)\,\d r}\mu_s(M)$.  Consequently,
$\mu_t/\mu_t(M)=\mu_s/\mu_s(M)$.
Thus $I_{t,s}$ preserves normalized measure.
Pullback by the embeddings then gives the last statement.
\end{proof}

The first-variation formulas characterize inverse mean curvature flow among
normal flows.

\begin{proposition}[Characterization of IMCF among normal flows]
\label{prop:imcf-selection}
Let $X_t:M^n\to(N^{n+1},\bar g)$, $t\in I$, be a smooth normal flow of
compact hypersurfaces, possibly with boundary, where $I$ is an interval
containing $0$,
\[
 \partial_t X=f_t\nu_t,
 \qquad f_t>0.
\]
Use the sign convention for which the second fundamental form $h_t$ is
positive on a strictly convex hypersurface, and let $g_t,H_t$, and $\d\vol_t$
denote the pullback metric, mean curvature, and Riemannian volume measure.  Then
\begin{equation}\label{eq:normal-first-variation}
 \partial_t g_t=2f_t h_t,
 \qquad
 \partial_t\d\vol_t=f_t H_t\,\d\vol_t.
\end{equation}
The following statements hold:
\begin{enumerate}
\item if $h_t\geq0$, every backward flow map is $1$-Lipschitz;
\item the backward maps preserve normalized volume for every $s<t$ if
and only if $f_t H_t$ is constant on each time slice, say
$f_t H_t=c(t)$; if $H_t>0$, then the time change
$\theta(t)=\int_0^t c(r)\,\d r$ turns the flow into
$\partial_\theta X=H^{-1}\nu$.
\end{enumerate}
\end{proposition}

\begin{proof}
In local coordinates, differentiating
$g_{ij}=\langle X_i,X_j\rangle_{\bar g}$ and using
$\overline\nabla_{X_i}\nu=h_i{}^kX_k$ gives
$\partial_t g_{ij}=2f_t h_{ij}$.  Taking one half of the $g_t$-trace yields
$\partial_t\d\vol_t=f_t H_t\,\d\vol_t$.  These computations are pointwise, so no
additional boundary term occurs when $M$ has boundary.
The first claim follows from Proposition~\ref{prop:mm-flow}.  For the second,
set $V_t:=\vol_{g_t}(M)$ and $\bar\mu_t:=V_t^{-1}\d\vol_t$.
By \eqref{eq:normal-first-variation},
\[
 V_t'=\int_M f_t H_t\,\d\vol_t
 =V_t\int_M f_t H_t\,\d\bar\mu_t.
\]
The quotient rule therefore yields 
\begin{equation}
 \partial_t\bar\mu_t=
 \left(f_t H_t-\int_M f_t H_t\,\d\bar\mu_t\right)\bar\mu_t.
\end{equation}
The backward maps preserve normalized volume for every $s<t$ precisely when
$\bar\mu_t$ is independent of $t$.  The last display shows that this is
equivalent to $f_t H_t$ being spatially constant.  Write this constant as
$c(t)$.  If $H_t>0$, then $c(t)>0$.
Equation \eqref{eq:normal-first-variation} also gives
$\d(\log V_t)/\d t=c(t)$,
which implies $\theta(t):=\log\left(V_t/V_0\right)=\int_0^t c(r)\,\d r$
is strictly increasing.  Let $t=t(\theta)$ denote its inverse and set
$\tilde X_\theta:=X_{t(\theta)}$.  Denote its mean curvature and unit
normal by $\tilde H_\theta$ and $\tilde\nu_\theta$, respectively.
Then
\begin{equation}
  \partial_\theta\tilde X_\theta
  =\frac{f_{t(\theta)}}{c(t(\theta))}\nu_{t(\theta)}
  =\tilde H_\theta^{-1}\tilde\nu_\theta,
\end{equation}
which is the inverse mean curvature flow equation.
\end{proof}

\begin{lemma}
\label{lem:inverse-bound}
Let $M^n$ and $N^n$ be compact connected smooth manifolds, possibly with
boundary, endowed with $C^1$ Riemannian metrics $g$ and $h$.  Let
$F:(M,g)\to(N,h)$ be a bi-Lipschitz homeomorphism satisfying
\[
 F_\#\frac{\d\vol_g}{\vol_g(M)}
 =\frac{\d\vol_h}{\vol_h(N)}.
\]
Set $b=\vol_h(N)/\vol_g(M)$ and $L=\Lip(F)$.  Then $J_F=b$ almost everywhere
with respect to $\vol_g$
and
\begin{equation}\label{eq:inverse-bound}
 \frac{b}{L^{n-1}}\dist_g(x,y)
 \leq \dist_h(Fx,Fy)
 \leq L\dist_g(x,y), \qquad \forall\,x,y\in M.
\end{equation}
The same conclusion holds when $M$ and $N$ are compact
full-dimensional geodesically convex subsets of $\Sph^n$, equipped with
their intrinsic length metrics and the restricted round volume measures. 
It also holds in dimension two, with Riemannian volume replaced by
$\Hh^2$, for bi-Lipschitz homeomorphisms between compact Alexandrov surfaces
without boundary. 
\end{lemma}

\begin{proof}
\noindent\emph{The smooth case.}
For every Borel set $E\subseteq M$, the pushforward identity
and the area formula yield
$\int_E J_F\,\d\vol_g=\vol_h(F(E))=b\vol_g(E)$
and hence $J_F=b$ almost everywhere with respect to $\vol_g$.  Denote
$G=F^{-1}$.  Since
bi-Lipschitz maps preserve null
sets, $F$ and $G$ are differentiable at corresponding points outside a null
set, with $\d G_{F(x)}=(\d F_x)^{-1}$.  If
$\sigma_1\geq\cdots\geq\sigma_n>0$ are the singular values of $\d F_x$, then
\begin{equation}
   \prod_{i=1}^n\sigma_i=b,
  \qquad \sigma_i\leq L,
  \qquad \norm{\d G_{F(x)}}=\sigma_n^{-1}\leq\frac{L^{n-1}}b.
\end{equation}
For fixed $z\in M$, define the Lipschitz function
$u_z:N\to\R$ by $u_z(y):=\dist_g(G(y),z)$.
Then $|\nabla u_z|_h\leq L^{n-1}/b$ at $\vol_h$-almost every point.
The standard $W^{1,\infty}$-to-Lipschitz principle for the intrinsic Riemannian
distance, obtained by applying the Euclidean result in interior and boundary
charts along an almost minimizing curve, shows that $u_z$ is
$L^{n-1}/b$-Lipschitz.  Taking $z=G(y')$ then proves
\eqref{eq:inverse-bound}.

\smallskip
\noindent\emph{The convex case.}
The boundaries of $M$ and $N$ have zero round volume, and $F$ maps their
interiors onto each other.  Applying the area formula and the pushforward
identity on the interiors gives $J_F=b$ almost everywhere.  Thus the
singular-value argument from the smooth case, with $G=F^{-1}$, yields
$\norm{\d G}\leq L^{n-1}/b$ at
$\d\vol_{\mathrm{can}}$-almost every point of $\operatorname{int}(N)$.
For fixed $z\in\operatorname{int}(M)$, define
$u_z:\operatorname{int}(N)\to\R$ by
$u_z(y):=\dist_M(G(y),z)$.  The same chain-rule argument gives
$|\nabla u_z|_{\mathrm{can}}\leq L^{n-1}/b$ at
$\d\vol_{\mathrm{can}}$-almost every point of $\operatorname{int}(N)$.
Since $N$ is geodesically convex, the $W^{1,\infty}$-to-Lipschitz argument
used above shows that $u_z$ is $L^{n-1}/b$-Lipschitz with respect to
$\dist_N$.  Taking $z=G(y')$ proves the lower bound for
$y,y'\in\operatorname{int}(N)$.  Extending the lower bound to all of $N$ by continuity proves
\eqref{eq:inverse-bound}. 

\smallskip
\noindent\emph{The Alexandrov case.}
For Alexandrov surfaces, the regular sets have full $\Hh^2$-measure and
admit countably many bi-Lipschitz strainer charts
\cite[Chapter~10]{BuragoBuragoIvanov}.  Since bi-Lipschitz maps preserve
$\Hh^2$-null sets, Rademacher's theorem
\cite[Corollary~2.14]{Bertrand}, applied coordinatewise in strainer charts,
gives a tangent differential $\d F_x$ at almost every regular point.  Choose
a countable bi-Lipschitz Borel partition
$M_i\subseteq\operatorname{Reg}(M)$ with parametrizations
$\phi_i:A_i\subseteq\mathbb R^2\to M_i$.  At almost every regular point,
the tangent spaces are Euclidean, so the chain rule and the definition of
Kirchheim's Jacobian \cite[Definition 5]{Kirchheim} show that the quotient of the area densities of
$F\circ\phi_i$ and $\phi_i$ is $J_F:=|\det\d F|$.

Thus for every Borel set $E\subseteq M$, applying Kirchheim's metric area formula \cite[Corollary~8]{Kirchheim} to $F\circ\phi_i$ and $\phi_i$, and then summing over $i$, we obtain 
\begin{equation}
 \int_E J_F\,\d\Hh_M^2=\Hh_N^2(F(E))=b\Hh_M^2(E).
\end{equation}
Hence $J_F=b$ almost everywhere with respect to $\Hh_M^2$.  Applying
Rademacher's theorem similarly to $G=F^{-1}$ and using the Euclidean chain
rule in strainer charts, we obtain at almost every corresponding pair
$y=F(x)$ of regular points, $\d G_y=(\d F_x)^{-1}$.  If
$\sigma_1\geq\sigma_2>0$ are the singular values of
$\d F_x:T_xM\to T_yN$, then
\begin{equation}
  \sigma_1\sigma_2=b,\qquad
  \sigma_1\leq L,\qquad
  \sigma_2^{-1}\leq\frac{L}{b}.
\end{equation}
Consequently, for fixed $z\in M$, the Lipschitz function
$u_z(y):=\dist_M(G(y),z)$ belongs to $W^{1,2}(N)\cap C(N)$.  The chain rule on the regular set yields 
$|\nabla u_z|\leq L/b$ at $\Hh_N^2$-almost every point.  By \cite[Theorem~7.1 and Proposition~7.2]{KuwaeMachigashiraShioya}, the
intrinsic-distance characterization, applied to $(b/L)u_z$, shows that
$u_z$ is $L/b$-Lipschitz.  In RCD terminology, this is the
Sobolev-to-Lipschitz principle; see also \cite{GigliHan} for the independence
of weak upper gradients from the Sobolev exponent.  Taking $z=G(y')$ proves
\eqref{eq:inverse-bound}. 
\end{proof}

The next proposition passes the finite-time maps to a terminal model.

\begin{proposition}[Endpoint contraction]
\label{prop:endpoint-contraction}
Let $M^n$ and $Z^n$ be compact connected manifolds of the same dimension,
either both closed or both with boundary.  Suppose the Riemannian
metrics $g_t$, $0\leq t<\tau$, depend continuously on $t$ and smoothly for
$t>0$, and satisfy
\begin{equation}\label{eq:endpoint-growth-hypotheses}
 g_t\geq g_s\quad s\leq t,
 \qquad
 \d\vol_{g_t}=e^t\d\vol_{g_0},
\end{equation}
where $0<\tau<\infty$.  Suppose there exist diffeomorphisms
$R_t:Z\to M$ for all $t$ sufficiently close to $\tau$ such that
$q_t:=R_t^*g_t\to q_\tau$ uniformly as quadratic forms, where $q_\tau$ is a
Riemannian metric on $Z$.  Then there exists a sequence $t_j\uparrow\tau$ such
that the maps $R_{t_j}$ converge uniformly to a bi-Lipschitz homeomorphism
\[
 R:(Z,q_\tau)\longrightarrow(M,g_0)
\]
which preserves normalized volume and is $1$-Lipschitz.  Moreover, if
$g_s\geq(1+\delta)g_0$
for some $s\in(0,\tau)$ and $\delta>0$, then
$\Lip(R)\leq(1+\delta)^{-1/2}<1$.  Set
$b:=\vol_{g_0}(M)/\vol_{q_\tau}(Z)=e^{-\tau}$ and $L:=\Lip(R)$.  Then
\begin{equation}\label{eq:endpoint-sharp-bilip}
 \frac{b}{L^{n-1}}\dist_{q_\tau}(x,y)
 \leq \dist_{g_0}(Rx,Ry)
 \leq L\dist_{q_\tau}(x,y),
 \qquad \forall\,x,y\in Z.
\end{equation}
\end{proposition}

\begin{proof}
\noindent\emph{Step 1: finite-time estimates.}
Since $q_t=R_t^*g_t$, the volume identity gives
\begin{equation}\label{eq:endpoint-finite-volume}
 R_t^*\d\vol_{g_0}=e^{-t}\d\vol_{q_t},
 \qquad
 (R_t)_\#\frac{\d\vol_{q_t}}{\vol_{q_t}(Z)}
 =\frac{\d\vol_{g_0}}{\vol_{g_0}(M)}.
\end{equation}
Moreover, by metric monotonicity in \eqref{eq:endpoint-growth-hypotheses},
$R_t^*g_0\leq q_t$, while
$\vol_{g_0}(M)/\vol_{q_t}(Z)=e^{-t}$.
Hence Lemma~\ref{lem:inverse-bound} gives
\begin{equation}\label{eq:endpoint-finite-bilip}
 e^{-t}\dist_{q_t}(x,y)
 \leq \dist_{g_0}(R_tx,R_ty)
 \leq \dist_{q_t}(x,y),
 \qquad \forall\,x,y\in Z.
\end{equation}
The uniform convergence $q_t\to q_\tau$ means that, for
$\varepsilon_t\downarrow0$,
$(1-\varepsilon_t)q_\tau\leq q_t\leq(1+\varepsilon_t)q_\tau$ near $\tau$.
The corresponding length bounds and the local determinant formula for volume
therefore imply
\begin{equation}
   \dist_{q_t}\to\dist_{q_\tau}
  \quad\text{uniformly on }Z\times Z,
  \qquad
  \frac{\d\vol_{q_t}}{\vol_{q_t}(Z)}
  \rightharpoonup
  \frac{\d\vol_{q_\tau}}{\vol_{q_\tau}(Z)} \quad \text{weakly}.
\end{equation}
Taking total masses in \eqref{eq:endpoint-finite-volume} and passing to the
limit yields $\vol_{q_\tau}(Z)=e^\tau\vol_{g_0}(M)$ and hence $b=e^{-\tau}$.

\smallskip
\noindent\emph{Step 2: passage to the endpoint.}
Choose a sequence $t_j\uparrow\tau$.  The preceding metric convergence and
\eqref{eq:endpoint-finite-bilip} give a uniform Lipschitz bound for
$R_{t_j}$ with respect to the fixed metrics $q_\tau$ and $g_0$.
By the Arzel\`a--Ascoli theorem, after passing to a subsequence,
$R_{t_j}$ converges uniformly to a map $R:Z\to M$.  Therefore,
passing to the limit in \eqref{eq:endpoint-finite-volume} and
\eqref{eq:endpoint-finite-bilip} yields
\begin{equation}
 b\,\dist_{q_\tau}(x,y)
 \leq \dist_{g_0}(Rx,Ry)
 \leq \dist_{q_\tau}(x,y),
 \quad \forall\,x,y\in Z, \qquad R_\#\frac{\d\vol_{q_\tau}}{\vol_{q_\tau}(Z)}
 =\frac{\d\vol_{g_0}}{\vol_{g_0}(M)}.
\end{equation}
The distance bound makes $R$ injective.  Its image is compact, and the
pushforward identity together with the full support of $\d\vol_{g_0}$ makes
it surjective.  Thus $R$ is a bi-Lipschitz homeomorphism.  Applying
Lemma~\ref{lem:inverse-bound} to $R$, with $L=\Lip(R)$, yields
\eqref{eq:endpoint-sharp-bilip}.

\smallskip
\noindent\emph{Step 3: strict contraction.}
Finally, if $g_s\geq(1+\delta)g_0$ for some $s\in(0,\tau)$, then for
$t\geq s$,
$R_t^*g_0\leq(1+\delta)^{-1}R_t^*g_s
\leq(1+\delta)^{-1}q_t$,
which implies $L\leq(1+\delta)^{-1/2}<1$.
\end{proof}

\section{Closed contractions}\label{sec:closed}

\subsection{Strictly convex spherical hypersurfaces}

We prove the closed part of Theorem~\ref{thm:extrinsic-unified}.  Let
$\Sigma_0\subseteq\Sph^{n+1}$ be as in part~(i), and let
$X_0:\Sigma_0\hookrightarrow\Sph^{n+1}$ be the inclusion.  By
\cite[Theorem~1.1]{doCarmoWarner}, $\Sigma_0$ is diffeomorphic to $\Sph^n$
and bounds a convex body in an open hemisphere.  Choose the outer unit
normal $\nu$ and use the convention
\[
 \overline\nabla_{X_i}\nu=h_i{}^kX_k,
 \qquad
 h_{ij}=g_{jk}h_i{}^k,
\]
where $X_i:=\d X_0(\partial_i)$ in local coordinates.
With this choice, the second fundamental form $h$ is positive definite.  We
write $H:=\operatorname{tr}_g h=g^{ij}h_{ij}
=\kappa_1+\cdots+\kappa_n$
for the mean curvature.

Consider the inverse mean curvature flow starting from $X_0$.
\begin{equation}\label{eq:spherical-imcf}
 \partial_t X=H_t^{-1}\nu_t,
 \qquad X(0,\cdot)=X_0.
\end{equation}
The curvature function $H(\kappa)=\kappa_1+\cdots+\kappa_n$ satisfies
Assumption~1.3(i) of \cite{MakowskiScheuer}; for its inverse concavity, see
\cite[Corollary~2.2]{Andrews}.  Hence
\cite[Theorem~1.4]{MakowskiScheuer}, applied with $F=H$ and $p=1$, yields a
unique solution on a finite interval $[0,\tau)$; see also
\cite{GerhardtSphere}.  Moreover, there exists $t_0\in(0,\tau)$ such that, for
every $t\in[t_0,\tau)$, $\Sigma_t=X_t(\Sigma_0)$ is the radial graph
$r=u_t(z)$, $z\in\Sph^n$, over a fixed equator in geodesic polar coordinates,
and
\begin{equation}\label{eq:spherical-graph-convergence}
  u_t\to\frac{\pi}{2}
  \quad\text{in }C^{1,\beta}(\Sph^n)
  \quad\text{as }t\uparrow\tau,
  \quad\text{for every }0<\beta<1.
\end{equation}

Write $X_t=X(t,\cdot)$ and set
$g_t:=X_t^*g_{\Sigma_t}$ and $\d\vol_t:=X_t^*\d\vol_{\Sigma_t}$.
By \cite[Lemma 4.6]{MakowskiScheuer}, strict convexity is preserved
along the flow.  Proposition~\ref{prop:imcf-selection}
therefore yields
\begin{equation}
  \partial_t g_t=2H_t^{-1}h_t>0,
  \qquad
  \partial_t\d\vol_t=\d\vol_t.
\end{equation}
Consequently,
\begin{equation}\label{eq:spherical-monotonicity}
 g_t\geq g_s\quad 0\leq s\leq t<\tau,
 \qquad
 \d\vol_t=e^t\d\vol_0,
 \qquad
 |\Sigma_t|=e^t|\Sigma_0|.
\end{equation}

\begin{proof}[Proof of Theorem~\ref{thm:extrinsic-unified}~(i)]
Identify the limiting equator isometrically with
$(\Sph^n,g_{\mathrm{can}})$.  For $t\in[t_0,\tau)$, let
$Q_t:\Sph^n\to\Sigma_t$ be the radial parametrization determined by $u_t$.
The spherical metric in polar coordinates is
$\d r^2+\sin^2r\,g_{\mathrm{can}}$, hence
\begin{equation}
   q_t:=Q_t^*g_{\Sigma_t}
  =\d u_t\otimes\d u_t+\sin^2(u_t)g_{\mathrm{can}}.
\end{equation}
By \eqref{eq:spherical-graph-convergence},
$q_t\to g_{\mathrm{can}}$ uniformly as quadratic
forms.  Consequently, $\dist_{q_t}\to\dist_{\mathrm{can}}$
uniformly, the density of $\d\vol_{q_t}$ with respect to
$\d\vol_{\mathrm{can}}$ tends uniformly to one, and
$|\Sigma_t|\to|\Sph^n|$.  Together with
\eqref{eq:spherical-monotonicity}, this yields
$b:=|\Sigma_0|/|\Sph^n|=e^{-\tau}$.

For $t\in[t_0,\tau)$, define
\[
 T_t:=X_t^{-1}\circ Q_t:\Sph^n\longrightarrow\Sigma_0.
\]
Then $T_t^*g_t=q_t$.  Fix $s\in(0,\tau)$.  Since $\Sigma_0$ is compact and
$g_s-g_0=\int_0^s 2H_r^{-1}h_r\,\d r>0$,
there exists $\delta>0$ such that
$g_s\geq(1+\delta)g_0$.
Applying Proposition~\ref{prop:endpoint-contraction} with $M=\Sigma_0$,
$Z=\Sph^n$, and $R_t=T_t$, we obtain a
bi-Lipschitz homeomorphism
\begin{equation*}
   T:(\Sph^n,g_{\mathrm{can}})\longrightarrow(\Sigma_0,g_0)
\end{equation*}
with $L:=\Lip(T)\leq(1+\delta)^{-1/2}<1$, satisfying
\begin{equation}
   \frac{b}{L^{n-1}}\dist_{\mathrm{can}}(z,z')
  \leq \dist_{g_0}(Tz,Tz')
  \leq L\dist_{\mathrm{can}}(z,z'),
  \quad \forall\,z,z'\in\Sph^n,\qquad T_\#\frac{\d\vol_{\mathrm{can}}}{|\Sph^n|}
  =\frac{\d\vol_0}{|\Sigma_0|},
\end{equation}
which proves Theorem~\ref{thm:extrinsic-unified}~(i).
\end{proof}

\subsection{E. Milman's conjecture for two-spheres}

The intrinsic theorem follows from the closed extrinsic result.  The
strict-curvature case uses Lu's solution of the spherical Weyl problem
\cite{Lu}.  The next lemma gives the regularity needed to apply the smooth
flow.  It is a local form of Pogorelov's regularity theorem for convex surfaces
in elliptic space; see
\cite[Chapter~VII, Section~4, Theorem~1]{PogorelovElliptic}.
\begin{lemma}\label{lem:automatic-regularity}
Let $(M^2,g)$ be a smooth Riemannian surface and let
$X\in C^2(M,\Sph^3)$ be a $C^2$ isometric immersion.  If $K_g>1$, then $X$ is
smooth.
\end{lemma}

\begin{proof}
Fix $p\in M$ and choose a $C^1$ unit normal on a simply connected
neighborhood $U$.  Regard $X|_U$ as a $C^2$ isometric immersion into
$\R^4$.  The low-regularity Gauss--Codazzi theorem
\cite[Theorem~5.2]{ChenLi} gives its Euclidean Gauss equation in the sense of
distributions.  Splitting the Euclidean second fundamental form into its
radial component and its component $h$ normal to $X(U)$ in $\Sph^3$ gives the
spherical Gauss equation distributionally.  Since $g$ is smooth and $h$ is
continuous, both sides are continuous; hence the identity holds pointwise:
\begin{equation}\label{eq:spherical-gauss-equation}
 \det(g^{-1}h)=K_g-1>0.
\end{equation}
After reversing the normal, shrink $U$ so that $h>0$ and $X|_U$ is a graph
in ambient normal coordinates.  By positive definiteness of $h$ and Taylor's
formula, after shrinking $U$ once more this graph lies strictly on one side
of each tangent plane.  Thus, $X(U)$ is a strictly convex cap contained in an open hemisphere $V$.  Let
$\pi:\Sph^3\to\mathbb E^3=\Sph^3/\{\pm1\}$ be the quotient to elliptic space.
The map $\pi|_V$ is a diffeomorphism, so $\pi(X(U))$ is a strictly convex cap
in $\mathbb E^3$.  If $F=\pi(X(U))$, then
\[
 (\pi\circ X)^*(g_{\mathbb E^3}|_F)=g|_U,
 \qquad
 K_F\circ(\pi\circ X)=K_g|_U>1.
\]

Pogorelov's local regularity theorem for convex surfaces in elliptic space
\cite[Chapter~VII, Section~4, Theorem~1]{PogorelovElliptic} states that $F$
is of class $C^{k-1}$ if the metric is $C^{k}$ for $k\geq5$.  Since $g$ is
smooth, $F$ is smooth.  The map
$\pi\circ X:(U,g)\to F$ is a $C^2$ local Riemannian isometry between
smooth surfaces.  After restricting to normal neighborhoods, it satisfies
\[
 \pi\circ X\circ\exp_p
 =\exp_{\pi(X(p))}\circ\d(\pi\circ X)_p,
\]
and hence is smooth.  Finally,
$X|_U=(\pi|_V)^{-1}\circ(\pi\circ X)$ is smooth.  Since $p$ is arbitrary,
the proof is complete.
\end{proof}

\begin{proof}[Proof of Theorem~\ref{thm:milman}]
\noindent\emph{Step 1: the strict-curvature case.}
Assume first that $K_g>1$.  Since $R_g=2K_g>2$,
\cite[Theorem~1.3]{Lu} gives a $C^2$ isometric embedding
\[
X:(M,g)\longrightarrow(\Sph^3,g_{\mathrm{can}}).
\]
Lemma~\ref{lem:automatic-regularity} implies that
$X\in C^\infty(M,\Sph^3)$.  Since $M$ is orientable, choose a global
normal.  Equation \eqref{eq:spherical-gauss-equation} implies that $h$ is
definite everywhere; connectedness and a reversal of the normal allow us to
take $h>0$.  Hence $X(M)$ is smooth, closed, connected, embedded, and
strictly convex.  Applying Theorem~\ref{thm:extrinsic-unified}~(i) to $X(M)$
and then composing
the resulting map with the isometry $X^{-1}:X(M)\to M$ yields the required map
with $L<1$.

\smallskip
\noindent\emph{Step 2: the limiting case.}
Now suppose only that $K_g\geq1$.  For $0<\varepsilon<1$, set
$g_\varepsilon=(1-\varepsilon)g$.
Then $K_{g_\varepsilon}=(1-\varepsilon)^{-1}K_g>1$,
$\d\vol_{g_\varepsilon}=(1-\varepsilon)\d\vol_g$, and
$a_\varepsilon=(1-\varepsilon)a$.
By Step~1, there exists a bi-Lipschitz homeomorphism
$T_\varepsilon:\Sph^2\to M$ with $L_\varepsilon:=\Lip(T_\varepsilon)<1$
which preserves normalized volume and satisfies
\[
 \frac{a_\varepsilon}{L_\varepsilon}\dist_{\mathrm{can}}(z,z')
 \leq\dist_{g_\varepsilon}(T_\varepsilon z,T_\varepsilon z')
 \leq L_\varepsilon\dist_{\mathrm{can}}(z,z').
\]
In particular,
\begin{equation}\label{eq:strict-approximation-map}
  a_\varepsilon \dist_{\mathrm{can}}(z,z')
  \leq \dist_{g_\varepsilon}(T_\varepsilon z,T_\varepsilon z')
  \leq \dist_{\mathrm{can}}(z,z'),\quad \forall\,z,z'\in\Sph^2, \qquad (T_\varepsilon)_\#
  \frac{\d\vol_{\mathrm{can}}}{|\Sph^2|}=\frac{\d\vol_g}{\vol_g(M)}.
\end{equation}
Since
$\dist_{g_\varepsilon}=\left(1-\varepsilon\right)^{1/2}\dist_g$, we obtain
\begin{equation}\label{eq:scaled-bilip}
  a\left(1-\varepsilon\right)^{\frac{1}{2}}\,\dist_{\mathrm{can}}(z,z')
  \leq \dist_g(T_\varepsilon z,T_\varepsilon z')
  \leq (1-\varepsilon)^{-\frac{1}{2}}\dist_{\mathrm{can}}(z,z'),
  \quad \forall\,z,z'\in\Sph^2.
\end{equation}

Choose a sequence $\varepsilon_j\downarrow0$.  The upper bound in
\eqref{eq:scaled-bilip} is uniformly bounded after discarding finitely
many terms.  By the Arzel\`a--Ascoli theorem, after passing to a subsequence,
$T_{\varepsilon_j}$ converges uniformly to a map $T:\Sph^2\to M$.  Hence,
passing to the limit in \eqref{eq:strict-approximation-map} and
\eqref{eq:scaled-bilip} yields
\begin{equation}
 a\,\dist_{\mathrm{can}}(z,z')
 \leq \dist_g(Tz,Tz')
 \leq \dist_{\mathrm{can}}(z,z'),
 \quad \forall\,z,z'\in\Sph^2, \qquad T_\#\frac{\d\vol_{\mathrm{can}}}{|\Sph^2|}
 =\frac{\d\vol_g}{\vol_g(M)}.
\end{equation}
The distance bounds make $T$ injective, and the measure identity makes it
surjective.  Thus $T$ is a bi-Lipschitz homeomorphism.  Applying
Lemma~\ref{lem:inverse-bound} with $L=\Lip(T)$ gives
\eqref{eq:intro-milman-bilip}.
\end{proof}

The smooth theorem extends to Alexandrov two-spheres through the conformal
approximation of Lin, Wang, and Xu \cite[Section~11]{LinWangXu}.

\begin{corollary}[Alexandrov two-spheres]\label{cor:alexandrov}
Let $(X,\dist_X)$ be a compact Alexandrov surface homeomorphic to $\Sph^2$ with
curvature bounded below by one.  Then there exists a bi-Lipschitz homeomorphism
\[
 T:(\Sph^2,g_{\mathrm{can}})\longrightarrow(X,\dist_X)
\]
such that
\[
 T_\#\frac{\d\vol_{\mathrm{can}}}{|\Sph^2|}
 =\frac{\Hh_X^2}{\Hh^2(X)},\qquad L:=\Lip(T)\leq1.
\]
With $a_X=\Hh^2(X)/|\Sph^2|$,
\begin{equation}\label{eq:alexandrov-bilip}
 \frac{a_X}{L}\dist_{\mathrm{can}}(z,z')
 \leq \dist_X(Tz,Tz')
 \leq L\dist_{\mathrm{can}}(z,z'),
 \qquad \forall\,z,z'\in\Sph^2.
\end{equation}
\end{corollary}

\begin{proof}
\noindent\emph{Step 1: smooth approximation.}
Using the unit round metric on $\mathbb{CP}^1$, identify it isometrically with
$(\Sph^2,g_{\mathrm{can}})$.  Let $\Phi:\Sph^2\to X$ be the conformal
homeomorphism of \cite[Notation~11.1]{LinWangXu}, and set
$\dist_X^\Phi(p,q):=\dist_X(\Phi p,\Phi q)$ and
$\mu_X^\Phi:=(\Phi^{-1})_\#\Hh_X^2$.
By the heat regularization in
\cite[Lemmas~11.7 and 11.8]{LinWangXu}, there exist smooth metrics $g_j$ on
$\Sph^2$ such that
\begin{equation}
 K_{g_j}\geq1,
 \qquad
\varepsilon_j:= \norm{\dist_{g_j}-\dist_X^\Phi}_{C^0(\Sph^2\times\Sph^2)}\to0,
 \qquad
 \norm{\d\vol_{g_j}-\mu_X^\Phi}_{\mathrm{TV}}\to0.
\end{equation}
Set $V_j=\vol_{g_j}(\Sph^2)$ and $a_j=V_j/|\Sph^2|$.  Then
$V_j\to\Hh^2(X)$, $a_j\to a_X$, and we have
\begin{equation}
   \frac{\d\vol_{g_j}}{V_j} \rightharpoonup
  \frac{\mu_X^\Phi}{\Hh^2(X)}
  \qquad\text{weakly}.
\end{equation}
By Theorem~\ref{thm:milman}, there exist bi-Lipschitz homeomorphisms
$S_j:\Sph^2\to\Sph^2$ with Lipschitz constants at most one.  Hence
\begin{equation}
 a_j\dist_{\mathrm{can}}(z,z')
 \leq \dist_{g_j}(S_jz,S_jz')
 \leq \dist_{\mathrm{can}}(z,z'),
 \quad \forall\,z,z'\in\Sph^2,
 \qquad
 (S_j)_\#\frac{\d\vol_{\mathrm{can}}}{|\Sph^2|}
 =\frac{\d\vol_{g_j}}{\vol_{g_j}(\Sph^2)}.
\end{equation}

\smallskip
\noindent\emph{Step 2: passage to the limit.}
Define $T_j=\Phi\circ S_j:\Sph^2\to X$.  By the definition of
$\varepsilon_j$,
\begin{equation}\label{eq:alexandrov-map-estimate}
 a_j\dist_{\mathrm{can}}(z,z')-\varepsilon_j
 \leq \dist_X(T_jz,T_jz')
 \leq \dist_{\mathrm{can}}(z,z')+\varepsilon_j,
 \qquad
 (T_j)_\#\frac{\d\vol_{\mathrm{can}}}{|\Sph^2|}
 =\Phi_\#\frac{\d\vol_{g_j}}{V_j}.
\end{equation}
Since $\varepsilon_j\to0$, the additive upper bound gives equicontinuity.
After passing to a subsequence, Arzel\`a--Ascoli gives uniform convergence to
a map $T:\Sph^2\to X$.  Passing to the limit in
\eqref{eq:alexandrov-map-estimate} yields
\begin{equation}
 a_X\dist_{\mathrm{can}}(z,z')
 \leq \dist_X(Tz,Tz')
 \leq \dist_{\mathrm{can}}(z,z'),
 \quad \forall\,z,z'\in\Sph^2,
 \qquad
 T_\#\frac{\d\vol_{\mathrm{can}}}{|\Sph^2|}
 =\frac{\Hh_X^2}{\Hh^2(X)}.
\end{equation}
The distance bounds and the measure identity show that $T$ is a bi-Lipschitz
homeomorphism.  The Alexandrov part of Lemma~\ref{lem:inverse-bound} now gives
\eqref{eq:alexandrov-bilip}.
\end{proof}

\begin{remark}
The constant one in Corollary~\ref{cor:alexandrov} is optimal.
The nonround spherical footballs in \cite[Example~17.3]{LinWangXu} have diameter
$\pi$, so every surjection from $\Sph^2$ onto them has Lipschitz constant at
least one.  Thus equality can occur for a nonround singular Alexandrov surface,
unlike in the smooth strictly convex case.
\end{remark}

\begin{corollary}[Metric and spectral consequences]
In the setting of Corollary~\ref{cor:alexandrov},
\begin{equation}
  \dist_{\mathrm{GH}}\bigl((X,\dist_X),(\Sph^2,g_{\mathrm{can}})\bigr)
  \leq\frac\pi2\left(1-\frac{a_X}{L}\right),\qquad
  \dist_{\mathrm{GH}}\bigl((X,\dist_X),(\Sph^2,a_Xg_{\mathrm{can}})\bigr)
  \leq\frac\pi2(L-\sqrt{a_X}).
\end{equation}
In particular, if $L=\sqrt{a_X}$, then $(X,\dist_X)$ is isometric to the round sphere
$(\Sph^2,a_Xg_{\mathrm{can}})$.

Moreover, let $0=\lambda_0(X)\leq\lambda_1(X)\leq\cdots$ denote the spectrum
of the canonical self-adjoint Laplacian of $(X,\dist_X,\Hh_X^2)$.  Then
\begin{equation}\label{eq:closed-spectrum-comparison}
 \frac{1}{L^2}\lambda_k(\Sph^2,g_{\mathrm{can}})
\leq\lambda_k(X)
\leq\left(\frac{L}{a_X}\right)^2
\lambda_k(\Sph^2,g_{\mathrm{can}}),\qquad k\geq0.
\end{equation}
\end{corollary}

\begin{proof}
Since $a_X/L\leq L\leq1$, equation \eqref{eq:alexandrov-bilip} gives
\begin{equation}
   \bigl|\dist_{\mathrm{can}}(z,z')-\dist_X(Tz,Tz')\bigr|
  \leq\left(1-\frac{a_X}{L}\right)
  \dist_{\mathrm{can}}(z,z').
\end{equation}
Since $T$ is surjective, its graph is a correspondence.  The correspondence
characterization of the Gromov--Hausdorff distance
\cite[Theorem~7.3.25]{BuragoBuragoIvanov} therefore yields
\begin{equation}
   \dist_{\mathrm{GH}}\bigl((X,\dist_X),(\Sph^2,g_{\mathrm{can}})\bigr)
  \leq\frac12\sup_{z,z'\in\Sph^2}
  \bigl|\dist_{\mathrm{can}}(z,z')-\dist_X(Tz,Tz')\bigr|
  \leq\frac\pi2\left(1-\frac{a_X}{L}\right).
\end{equation}

On the other hand, since $\dist_{a_Xg_{\mathrm{can}}}=\sqrt{a_X}\dist_{\mathrm{can}}$,
\eqref{eq:alexandrov-bilip} also gives
\begin{equation}\label{eq:closed-equal-volume-bilip}
 \frac{\sqrt{a_X}}{L}\dist_{a_Xg_{\mathrm{can}}}(z,z')
 \leq\dist_X(Tz,Tz')
 \leq\frac{L}{\sqrt{a_X}}
 \dist_{a_Xg_{\mathrm{can}}}(z,z').
\end{equation}
Using the graph of $T$ as a correspondence gives
\begin{equation}
   \dist_{\mathrm{GH}}\bigl((X,\dist_X),
  (\Sph^2,a_Xg_{\mathrm{can}})\bigr)
  \leq\frac12\left(\frac{L}{\sqrt{a_X}}-1\right)
  \operatorname{diam}(\Sph^2,a_Xg_{\mathrm{can}})
  =\frac\pi2(L-\sqrt{a_X}).
\end{equation}
If $L=\sqrt{a_X}$, then
\eqref{eq:closed-equal-volume-bilip} shows that $T$ is an isometry.

The normalized measure identity yields
$T_\#\d\vol_{\mathrm{can}}=a_X^{-1}\Hh_X^2$.
Moreover, \eqref{eq:alexandrov-bilip} gives
$\Lip(T)\leq L$ and $\Lip(T^{-1})\leq L/a_X$.
Applying Milman's metric contraction principle
\cite[Proposition~3.1 and Section~3.3]{Milman} to these two maps, together
with the identification of the relaxed Lipschitz energy with the canonical
Dirichlet form on Alexandrov spaces \cite{KuwaeMachigashiraShioya} yields
the two spectral bounds in \eqref{eq:closed-spectrum-comparison}.
\end{proof}

Since $L\leq1$, the lower estimate in
\eqref{eq:closed-spectrum-comparison} recovers the spectral bound of
\cite[Theorem~1.8]{LinWangXu} in the smooth case.
We prove rigidity in the equality case of Milman's metric contraction
principle \cite[Proposition~3.1]{Milman} for smooth two-spheres.  The result
applies to every Lipschitz map onto a smooth Riemannian two-sphere that
preserves normalized volume; neither injectivity nor a curvature assumption is
required.  Both the smoothness assumption and the restriction to dimension
two are essential.
Nonround spherical footballs provide counterexamples in the singular setting
\cite[Example~17.3]{LinWangXu}, while Berger spheres show that the analogous
equality rigidity fails already in dimension three
\cite[Proposition~3.9]{LauretBerger}.

\begin{theorem}[Spectral rigidity for measure-preserving Lipschitz maps]
\label{thm:spectral-rigidity}
Let $(M^2,g)$ be a smooth Riemannian two-sphere and let
$T:(\Sph^2,g_{\mathrm{can}})\longrightarrow(M,g)$ be a Lipschitz map
satisfying
\[
 T_\#\frac{\d\vol_{\mathrm{can}}}{|\Sph^2|}
 =\frac{\d\vol_g}{\vol_g(M)}.
\]
Set $a:=\vol_g(M)/|\Sph^2|$ and $L:=\Lip(T)$.  Then $a\leq L^2$ and
\begin{equation}\label{eq:spectral-lower-L}
 \lambda_k(M,g)\geq
 L^{-2}\lambda_k(\Sph^2,g_{\mathrm{can}}),
 \qquad k\geq0.
\end{equation}
Moreover, the following statements are equivalent:
\begin{enumerate}
\item equality in \eqref{eq:spectral-lower-L} holds for some
$k\geq1$;
\item $a=L^2$;
\item $T:(\Sph^2,L^2g_{\mathrm{can}})\longrightarrow(M,g)$ is an isometry;
\item equality in \eqref{eq:spectral-lower-L} holds for every $k\geq0$.
\end{enumerate}
\end{theorem}

\begin{proof}
\noindent\emph{Step 1: the spectral bound.}
The pushforward identity and full support of the Riemannian volume measures
imply that $T$ is
surjective, so $L>0$.  Set $\tilde g=L^{-2}g$ and denote
$\mu_0=\d\vol_{\mathrm{can}}/|\Sph^2|$ and
$\tilde\mu=\d\vol_{\tilde g}/\vol_{\tilde g}(M)$.
Then $T:(\Sph^2,g_{\mathrm{can}})\to(M,\tilde g)$ is $1$-Lipschitz and
$T_\#\mu_0=\tilde\mu$.  The area formula gives
$b:=\vol_{\tilde g}(M)/|\Sph^2|=a/L^2\leq1$.
Indeed, surjectivity and the area formula \cite{Kirchheim} give
$\vol_{\tilde g}(M)\leq\int_{\Sph^2}J_T\,\d\vol_{\mathrm{can}}$, while
$J_T\leq1$ almost everywhere.
Milman's contraction principle \cite[Proposition~3.1]{Milman} then gives
\begin{equation}
  \lambda_j(\Sph^2,g_{\mathrm{can}})
  \leq\lambda_j(M,\tilde g)
  =L^2\lambda_j(M,g).
\end{equation}

\smallskip
\noindent\emph{Step 2: an extremizing eigenfunction.}
Assume equality in \eqref{eq:spectral-lower-L} for some $k\geq1$.
Since $T_\#\mu_0=\tilde\mu$, the pullback
operator
\[
U:L^2(M,\tilde\mu)\longrightarrow L^2(\Sph^2,\mu_0),
\qquad Uf=f\circ T,
\]
is a linear isometric embedding.
Set
$\Lambda=\lambda_k(\Sph^2,g_{\mathrm{can}})=\lambda_k(M,\tilde g)>0$.
Let
\[
E:=
\bigoplus_{\substack{\lambda\in\operatorname{Spec}(-\Delta_{\tilde g}); \,
    \lambda\leq\Lambda}}
\ker(-\Delta_{\tilde g}-\lambda\operatorname{Id}),
\qquad
F:=
\bigoplus_{\substack{\lambda\in\operatorname{Spec}(-\Delta_{g_{\mathrm{can}}});\,
    \lambda<\Lambda}}
  \ker(-\Delta_{g_{\mathrm{can}}}-\lambda\operatorname{Id}).
\]
Then $\dim E\geq k+1$ and $\dim F\leq k$, and hence there exists
$0\neq f\in E$ such that $u=Uf\in F^\perp$.  Let $\mathcal E_0$ and
$\mathcal E_{\tilde g}$ denote the Dirichlet energies associated with
$(g_{\mathrm{can}},\mu_0)$ and $(\tilde g,\tilde\mu)$.
Since $T$ is $1$-Lipschitz, the Sobolev chain rule yields
$\mathcal E_0(Uw)\leq\mathcal E_{\tilde g}(w)$ for any
$w\in W^{1,2}(M)$.  Hence, by the definitions of
$E$ and $F$ and the $L^2$ isometry of $U$, we obtain
\begin{equation}
     \Lambda\norm{u}_{L^2(\mu_0)}^2
    \leq\mathcal E_0(u)
  \leq\mathcal E_{\tilde g}(f)
  \leq\Lambda\norm{f}_{L^2(\tilde\mu)}^2
  =\Lambda\norm{u}_{L^2(\mu_0)}^2.
\end{equation}
Hence all inequalities are equalities.  Consequently,
\begin{equation}
   -\Delta_{g_{\mathrm{can}}}u=\Lambda u,\qquad
  -\Delta_{\tilde g}f=\Lambda f,\qquad
  u=f\circ T,
  \qquad
  |\nabla u|_{g_{\mathrm{can}}}\overset{\text{a.e.}}{=}|\nabla f|_{\tilde g}\circ T.
\end{equation}
The first two conclusions follow from the spectral decompositions of $u$ and
$f$.  The gradient identity follows because the Sobolev chain rule gives a
pointwise inequality whose integrals are equal, using $T_\#\mu_0=\tilde\mu$.
Elliptic regularity gives $u,f\in C^\infty$.  Since both sides of the last
identity are continuous, that identity holds pointwise.

\smallskip
\noindent\emph{Step 3: regular level sets.}
Set
$\operatorname{Crit}(u):=\{x\in\Sph^2:\nabla u(x)=0\}$ and
$\operatorname{Crit}(f):=\{y\in M:\nabla f(y)=0\}$.
At almost every $x\in\Sph^2\setminus\operatorname{Crit}(u)$ at which $T$
is differentiable, set
\[
n_u(x):=\frac{\nabla u(x)}{|\nabla u(x)|},
\qquad
n_f(T(x)):=\frac{\nabla f(T(x))}{|\nabla f(T(x))|},
\qquad
A_x:=\d T_x.
\]
By the chain rule applied to $u=f\circ T$, we have
$A_x^*n_f(T(x))=n_u(x)$, and $\norm{A_x}\leq1$ then yields
$A_x n_u(x)=n_f(T(x))$: indeed,
$\langle A_x n_u,n_f\rangle=1$ and $|A_x n_u|\leq1$.
Differentiating
$|\nabla u|_{g_{\mathrm{can}}}=|\nabla f|_{\tilde g}\circ T$ in the $n_u$
direction gives
\begin{equation}
  \begin{aligned}
    (\nabla^2u)_x\bigl(n_u(x),n_u(x)\bigr)
    &=\d\bigl(|\nabla u|_{g_{\mathrm{can}}}\bigr)_x\bigl(n_u(x)\bigr)\\
    &=\d\bigl(|\nabla f|_{\tilde g}\bigr)_{T(x)}
    \bigl(A_x n_u(x)\bigr)=(\nabla^2f)_{T(x)}
    \bigl(n_f(T(x)),n_f(T(x))\bigr).
  \end{aligned}
\end{equation}

With the Gauss--Bonnet sign convention, the signed geodesic curvature of a
regular level set of $v$ is denoted by
$\kappa_v=(\Delta v-\nabla^2v(n_v,n_v))/|\nabla v|$.  Then
$u=f\circ T$, the two eigenvalue equations, and the preceding Hessian
identity give $\kappa_u=\kappa_f\circ T$ almost everywhere on
$\Sph^2\setminus\operatorname{Crit}(u)$.
For any $\eta\in C_c(\R)$, the coarea formula and
$\d\vol_{\tilde g}=b\,T_\#\d\vol_{\mathrm{can}}$ yield
\begin{equation}
   \int_{\R} \eta(t)\int_{\{f=t\}} \kappa_f\,\d s_{\tilde g}\,\d t
  =\int_M \eta(f)\kappa_f|\nabla f|_{\tilde g}\,\d\vol_{\tilde g}
  =b\int_{\R} \eta(t)\int_{\{u=t\}} \kappa_u\,\d s_{g_{\mathrm{can}}}\,\d t.
\end{equation}
Thus, for almost every common regular value $t$ of $u$ and $f$,
\begin{equation}\label{eq:level-curvature-pushforward}
 \int_{\{f=t\}} \kappa_f\,\d s_{\tilde g}
 =b\int_{\{u=t\}} \kappa_u\,\d s_{g_{\mathrm{can}}}.
\end{equation}

\smallskip
\noindent\emph{Step 4: Gauss--Bonnet and the volume ratio.}
Let $\Omega_t=\{u<t\}$ and $D_t=\{f<t\}$.  Combining the Gauss--Bonnet
formulas for $\Omega_t$ and $D_t$ with $K_{g_{\mathrm{can}}}=1$ and
\eqref{eq:level-curvature-pushforward}, we obtain
\begin{equation}\label{eq:sublevel-gauss-bonnet}
 \mathcal G(t):=
 \int_{D_t} K_{\tilde g}\,\d\vol_{\tilde g}
 -b\,\vol_{g_{\mathrm{can}}}(\Omega_t)
 =2\pi\bigl(\chi(D_t)-b\chi(\Omega_t)\bigr).
\end{equation}
The critical set of a nonconstant eigenfunction on a smooth closed surface
has zero area \cite[Remark~4.4]{BakriCritical}.  This applies to both $u$ and
$f$; alternatively,
$T^{-1}(\operatorname{Crit}(f))=\operatorname{Crit}(u)$ and the pushforward
identity transfer the null-set statement from $u$ to $f$.
The coarea formula therefore yields
\begin{equation}
     \mathcal G(t)=\int_{-\infty}^t
  \left(
  \int_{\{f=s\}} \frac{K_{\tilde g}}{|\nabla f|}\,\d s_{\tilde g}
  -b\int_{\{u=s\}} \frac{1}{|\nabla u|}\,\d s_{g_{\mathrm{can}}}
  \right)\d s,
\end{equation}
which implies that $\mathcal G$ is absolutely continuous.  By Sard's theorem,
the union of the critical values of $u$ and $f$ has measure zero.  On each component of the complement of this union, the normalized gradient
flows $\nabla u/|\nabla u|^2$ and $\nabla f/|\nabla f|^2$ identify, respectively, the sublevel sets of $u$ and $f$ at different levels.  Hence the two Euler characteristics in
\eqref{eq:sublevel-gauss-bonnet} are constant there.  Since that identity holds
for almost every such value and $\mathcal G$ is continuous, $\mathcal G$ is
constant on each component.  Hence $\mathcal G'=0$ almost everywhere and
$\mathcal G$ is constant on $\R$.  For all sufficiently small $t$, both
sublevel sets are empty, and hence
$\mathcal G(t)=0$.  For all sufficiently large $t$, we have
$D_t=M$ and $\Omega_t=\Sph^2$, and Gauss--Bonnet gives
\begin{equation}
   \mathcal G(t)=
  \int_M K_{\tilde g}\,\d\vol_{\tilde g}
  -b|\Sph^2|=4\pi(1-b).
\end{equation}
Hence $b=1$ and then $a=L^2$.  This proves (i)$\Rightarrow$(ii).

\smallskip
\noindent\emph{Step 5: rigidity at equal volume.}
Assume now only (ii), so that $b=1$.  Suppose that
$T(x_1)=T(x_2)=y$ with $x_1\neq x_2$.  For all sufficiently small $r$, the
balls $B_r(x_1)$ and $B_r(x_2)$ are disjoint and contained in
$T^{-1}(B_r(y))$.  The pushforward identity then implies
\begin{equation}
  2\pi r^2+o(r^2)
  \leq\vol_{g_{\mathrm{can}}}\bigl(T^{-1}(B_r(y))\bigr)
  =\vol_{\tilde g}(B_r(y))
  =\pi r^2+o(r^2),
\end{equation}
which is a contradiction.  Thus $T$ is injective and hence a homeomorphism.
The area formula gives $J_T=1$ almost everywhere.  Since
$\norm{\d T}\leq1$, both singular values of $\d T$ are therefore one, so
$\d T$ is a linear isometry almost everywhere.  Give $M$ the orientation for
which the homeomorphism $T$ has degree one.  At every differentiability point
where $\d T$ is invertible, rescaling in oriented coordinate charts converges
uniformly on the boundary of a small disk to the linear map $\d T$.  Homotopy
invariance of the local degree then gives $\det(\d T)>0$.  Thus
$\d T\in\mathrm{SO}(2)$ almost everywhere.  By Reshetnyak rigidity
\cite[Theorem~1.1]{KupfermanMaorShachar}, $T$ is a smooth local isometry.
Since $T$ is a homeomorphism, it is a global isometry
$T:(\Sph^2,g_{\mathrm{can}})\longrightarrow(M,\tilde g)$, equivalently
$T:(\Sph^2,L^2g_{\mathrm{can}})\longrightarrow(M,g)$ is an isometry.

This proves (ii)$\Rightarrow$(iii).  If (iii) holds, the scaling law for the
Laplacian yields $\lambda_k(M,g)
=L^{-2}\lambda_k(\Sph^2,g_{\mathrm{can}})$ for every $k\geq0$.
Thus (iii)$\Rightarrow$(iv), while (iv)$\Rightarrow$(i) is immediate.
\end{proof}

\begin{remark}\label{rem:spectral-upper-rigidity}
  If $T$ is bi-Lipschitz, then the upper bound holds 
  \[
  \lambda_k(M,g)\leq
  \left(\frac{L}{a}\right)^2
  \lambda_k(\Sph^2,g_{\mathrm{can}}),
  \qquad k\geq0.
  \]
  This upper bound has the same rigidity: equality for some $k\geq1$ holds if
  and only if $a=L^2$, or equivalently,
  $T:(\Sph^2,L^2g_{\mathrm{can}})\to(M,g)$ is an isometry.  The level-set and
  Gauss--Bonnet argument in the proof of
  Theorem~\ref{thm:spectral-rigidity} applies equally to an arbitrary smooth
  Riemannian two-sphere $(N,h)$ as the source.  Indeed, in
  \eqref{eq:sublevel-gauss-bonnet} one replaces
  $b\,\vol_{g_{\mathrm{can}}}(\Omega_t)$ by
  $b\int_{\Omega_t}K_h\,\d\vol_h$.  Both total curvature integrals are
  $4\pi$, so the same endpoint argument gives $b=1$; the critical-set
  nullity needed for coarea is supplied by \cite[Remark~4.4]{BakriCritical}.
  Applying this observation to $T^{-1}$ proves the asserted upper rigidity.
  For
  comparison, volume-normalized upper bounds of a different nature are proved in
  \cite{KarpukhinNadirashviliPenskoiPolterovich}.
\end{remark}

Combining Theorem~\ref{thm:milman} and
Theorem~\ref{thm:spectral-rigidity} gives an alternative proof of the smooth
spectral rigidity theorem of Lin, Wang, and Xu
\cite[Theorem~1.8]{LinWangXu}.
\begin{corollary}[Spectral rigidity of Lin--Wang--Xu]
Under the assumptions of Theorem~\ref{thm:milman}, the equality
\[
 \lambda_k(M,g)=\lambda_k(\Sph^2,g_{\mathrm{can}})
\]
for some $k\geq1$ holds if and only if $(M,g)$ is isometric to the unit round
sphere.
\end{corollary}

\section{Free-boundary contractions}\label{sec:free-boundary}
We prove the free-boundary part of
Theorem~\ref{thm:extrinsic-unified}.  A compact embedded hypersurface
$\Sigma^n\subseteq\overline{\B^{n+1}}$ has free boundary if
\begin{equation}\label{eq:free-boundary-condition}
 \Sigma\cap\Sph^n=\partial\Sigma,
 \qquad
 \langle\nu,x\rangle=0\quad\text{on }\partial\Sigma.
\end{equation}
Thus $\Sigma$ meets the boundary sphere
$\Sph^n=\partial\overline{\B^{n+1}}$ orthogonally along $\partial\Sigma$.

\begin{proof}[Proof of Theorem~\ref{thm:extrinsic-unified}~(ii)]
\noindent\emph{Step 1: the free-boundary flow.}
Choose a smooth parametrization
$X_0:\overline{\B^n}\to\Sigma$ and the normal for which its second
fundamental form $h_0$ is positive definite.
Let $\gamma:(-\varepsilon,0]\to\Sigma$ be a $C^1$ curve with
$\gamma(0)\in\partial\Sigma$.  Since
$\Sigma\subseteq\overline{\B^{n+1}}$, we have
\begin{equation}
   2\langle\dot\gamma(0),\gamma(0)\rangle
  =\lim_{r\uparrow0}
  \frac{|\gamma(0)|^2-|\gamma(r)|^2}{-r}\geq0.
\end{equation}
This verifies the one-sided
hypothesis of Lambert and Scheuer; the remaining boundary conditions follow
from \eqref{eq:free-boundary-condition}.  By
\cite[Theorem~1.1]{LambertScheuer}, the inverse mean curvature flow
\begin{equation}
   \partial_t X=H_t^{-1}\nu_t,\qquad X(0,\cdot)=X_0,
\end{equation}
exists on a finite maximal interval $[0,\tau)$ and converges to a flat unit
disk.  Moreover, the strict convexity is preserved by
\cite[Proposition~6.6]{LambertScheuer}.

Use $X_0$ to identify $\overline{\B^n}$ with
$\Sigma_0:=\Sigma$, and still write $X_t:\Sigma_0\to\Sigma_t$ for the resulting
flow embeddings.  Set
$g_t=X_t^*g_{\Sigma_t}$ and
$\d\vol_t=X_t^*\d\vol_{\Sigma_t}$.
For $t\in(0,\tau)$, Proposition~\ref{prop:imcf-selection} gives
\begin{equation}
   \partial_t g_t=2H_t^{-1}h_t>0,
   \qquad
   \partial_t\d\vol_t=\d\vol_t,
\end{equation}
while $X_t\to X_0$ in $C^1$ as $t\downarrow0$.  Then
\begin{equation}
  g_t\geq g_s, \quad 0\leq s\leq t<\tau,
  \qquad
  \d\vol_t=e^t\d\vol_0.
\end{equation}

\smallskip
\noindent\emph{Step 2: the terminal flat disk.}
Let $Q_t(x)=f(x,u_t(x))$, $x\in\overline{\B^n}$, be the M\"obius graph
parametrizations used in the proof of
\cite[Theorem~1.1]{LambertScheuer}; see
\cite[Section~5]{LambertScheuer}.
After an ambient rotation, identify the limiting flat disk with
$\overline{\B^n}$.  By \cite[Remark~7.4]{LambertScheuer},
\begin{equation}
  u_t\to1
  \quad\text{in }C^{1,\beta}(\overline{\B^n})
  \quad\text{as }t\uparrow\tau,
  \quad\text{for every }0<\beta<1.
\end{equation}
Since $f(x,1)=x$ by the definition of $f$, it follows that
\begin{equation}
  Q_t\to\operatorname{id}_{\overline{\B^n}}
  \quad\text{in }C^{1,\beta}(\overline{\B^n})
  \quad\text{as }t\uparrow\tau,
  \quad\text{for every }0<\beta<1.
\end{equation}
Consequently, for $q_t:=Q_t^*g_{\Sigma_t}$,
$q_t\to g_{\mathrm{Euc}}$ uniformly as quadratic forms.  In
particular, $|\Sigma_t|\to\omega_n$, and hence
$b:=|\Sigma_0|/\omega_n=e^{-\tau}$.

\smallskip
\noindent\emph{Step 3: the endpoint map.}
For $t\in(0,\tau)$, define
\[
T_t:=X_t^{-1}\circ Q_t:\overline{\B^n}\longrightarrow\Sigma_0.
\]
Then $T_t^*g_t=q_t$.  Fix $s\in(0,\tau)$.  Since $\Sigma_0$ is compact and
$g_s-g_0=\int_0^s 2H_r^{-1}h_r\,\d r>0$,
there exists $\delta>0$ such that
$g_s\geq(1+\delta)g_0$.  Applying
Proposition~\ref{prop:endpoint-contraction} with $M=\Sigma_0$,
$Z=\overline{\B^n}$, and $R_t=T_t$, we obtain a
bi-Lipschitz homeomorphism
\begin{equation*}
    T:(\overline{\B^n},g_{\mathrm{Euc}})\longrightarrow(\Sigma_0,g_0)
\end{equation*}
with $L:=\Lip(T)\leq(1+\delta)^{-1/2}<1$, satisfying
\begin{equation}
  \frac{b}{L^{n-1}}|z-z'|
  \leq \dist_{g_0}(Tz,Tz')
  \leq L|z-z'|,
  \qquad \forall\,z,z'\in\overline{\B^n}, \qquad T_\#\frac{\d x}{\omega_n}
  =\frac{\d\vol_0}{|\Sigma_0|},
\end{equation}
which proves Theorem~\ref{thm:extrinsic-unified}~(ii).
\end{proof}

\begin{corollary}[Volume and spectral consequences]
\label{cor:free-consequences}
In the free-boundary setting of Theorem~\ref{thm:extrinsic-unified},
$|\Sigma|\leq L^n\omega_n<\omega_n$.  Moreover, for both Dirichlet and
Neumann boundary conditions,
\begin{equation}
 \lambda_k^{D/N}(\Sigma,g_\Sigma)
 \geq L^{-2}\lambda_k^{D/N}(\B^n,g_{\mathrm{Euc}}).
\end{equation}
Here $k=0,1,\ldots$ in the Neumann case and $k=1,2,\ldots$ in the Dirichlet case.
\end{corollary}

\begin{proof}
  The area formula gives $|\Sigma|\leq L^n\omega_n$.  Since $T$ is a
  homeomorphism of
  manifolds with boundary, $T(\partial\B^n)=\partial\Sigma$.  Applying
  Milman's contraction principle with boundary
  \cite[Theorem~3.6]{Milman} to both the Neumann and
  Dirichlet spectra gives the stated inequality.
\end{proof}

\begin{remark}
  Lambert and Scheuer proved that $|\Sigma|\leq\omega_n$ for weakly convex
  free-boundary hypersurfaces, with equality only for the flat disk
  \cite{LambertScheuerInequality}.  In the strictly convex setting, the
  volume estimate here is strengthened by the factor $L<1$, where $L$ is
  the Lipschitz constant of the contraction.  The strict inequality is not
  uniform.  For example, for $c>1$, the portion in
  $\overline{\B^{n+1}}$ of the sphere centered at $ce_{n+1}$ with radius
  $\sqrt{c^2-1}$ is a strictly convex free-boundary cap; the orthogonality
  follows from $c x_{n+1}=1$ along its boundary.  As $c\to\infty$, these caps
  converge in $C^1$ to the flat unit disk, in which case
  $|\Sigma|/\omega_n\to1$ and hence necessarily $L\to1$.
\end{remark}

The intrinsic disk result follows from the free-boundary contraction,
Koerber's isometric embedding theorem \cite{Koerber}, and a conformal
approximation.

\begin{proof}[Proof of Theorem~\ref{thm:disk}]
	\noindent\emph{Step 1: the positive-curvature case.}
Assume first that $K_g>0$.  Fix $\alpha\in(0,1)$.  Since $g$ is smooth,
\cite[Theorem~1.1]{Koerber} gives, for every integer $k\geq4$, a
$C^{k+1,\alpha}$ isometric embedding
\[
F_k:(\overline{\B^2},g)
\longrightarrow(\overline{\B^3},g_{\mathrm{Euc}})
\]
with free boundary on $\Sph^2$.  Fix $F_4$.  By the uniqueness statement in
\cite[Theorem~1.1]{Koerber}, for each $k\geq4$ there exists $Q_k\in O(3)$ such
that
$F_4=Q_k\circ F_k$.  Thus $F_4$ is of class $C^{k+1,\alpha}$ for every
$k\geq4$, and hence $F:=F_4$ is smooth.

Set $\Sigma=F(\overline{\B^2})$.  Since $\overline{\B^2}$ is orientable,
choose a global unit normal along $F$ and denote the corresponding second
fundamental form on $\overline{\B^2}$ by $h$.  The Gauss equation yields
$\det(g^{-1}h)=K_g>0$.
Thus $h$ is definite everywhere; connectedness and a reversal of the normal
allow us to take $h>0$.  Hence $\Sigma$ is a smooth, connected, embedded,
strictly convex free-boundary surface of disk type.  Apply
Theorem~\ref{thm:extrinsic-unified}~(ii) to $\Sigma$ and compose the map with
the isometry $F^{-1}:\Sigma\to(\overline{\B^2},g)$ to obtain the required map
with $L<1$.

\smallskip
	\noindent\emph{Step 2: the nonnegative-curvature case.}
Now suppose only that $K_g\geq0$.  Pulling back by a smooth diffeomorphism of
$\overline{\B^2}$, write $g=E_0^2g_{\mathrm{Euc}}$ in global isothermal
coordinates; see \cite[Section~2]{Koerber}.  Choose a smooth metric
$g_1=E_1^2g_{\mathrm{Euc}}$ with $K_{g_1}>0$ and
$k_{\partial\overline{\B^2},g_1}=1$, for instance the metric of a strictly
convex free-boundary spherical cap.  For $0<\varepsilon<1$, set
\[
 E_\varepsilon
 =\frac{E_0E_1}{(1-\varepsilon)E_1+\varepsilon E_0},
 \qquad
 g_\varepsilon=E_\varepsilon^2g_{\mathrm{Euc}}.
\]
Then $g_\varepsilon\to g$ smoothly.  For $g_E=E^2g_{\mathrm{Euc}}$, our sign
convention gives
$k_{\partial\B^2,g_E}
=E^{-1}\bigl(1+\partial_{\eta_{\mathrm{Euc}}}\log E\bigr)$.
Equivalently, $k_{\partial\B^2,g_E}=1$ precisely when
$\partial_{\eta_{\mathrm{Euc}}}(E^{-1})=E^{-1}-1$.  Since
$E_\varepsilon^{-1}=(1-\varepsilon)E_0^{-1}
+\varepsilon E_1^{-1}$, this linear boundary condition gives
$k_{\partial\overline{\B^2},g_\varepsilon}=1$.  Moreover, writing
$D_\varepsilon=(1-\varepsilon)E_1+\varepsilon E_0$, the computation in
\cite[Lemma~2.1]{Koerber} gives
\begin{equation}
  \begin{aligned}
    -\Delta_{\mathrm{Euc}}\log E_\varepsilon
    &=
    \frac{(1-\varepsilon)E_1E_0^2K_g
      +\varepsilon E_0E_1^2K_{g_1}}
    {D_\varepsilon} \\
    &\quad+
    \frac{\varepsilon(1-\varepsilon)}{D_\varepsilon^2}
    \left|
    \sqrt{\frac{E_1}{E_0}}\,\nabla_{\mathrm{Euc}}E_0
    -
    \sqrt{\frac{E_0}{E_1}}\,\nabla_{\mathrm{Euc}}E_1
    \right|_{g_{\mathrm{Euc}}}^{\,2}
    >0.
  \end{aligned}
\end{equation}
Thus $K_{g_\varepsilon}>0$.  By the strict case, there exists a
bi-Lipschitz homeomorphism
$T_\varepsilon:\overline{\B^2}\to\overline{\B^2}$ such that,
with
$a_\varepsilon=\vol_{g_\varepsilon}(\B^2)/\omega_2$ and
$L_\varepsilon=\Lip(T_\varepsilon)<1$,
\begin{equation}\label{eq:disk-approx-bilip}
   \frac{a_\varepsilon}{L_\varepsilon}|z-z'|
  \leq\dist_{g_\varepsilon}(T_\varepsilon z,T_\varepsilon z')
  \leq L_\varepsilon|z-z'|, \quad \forall\,z,z'\in\overline{\B^2},
  \qquad
  (T_\varepsilon)_\#\frac{\d x}{\omega_2}
  =\frac{\d\vol_{g_\varepsilon}}{\vol_{g_\varepsilon}(\B^2)}.
\end{equation}
Choose a sequence $\varepsilon_j\downarrow0$.  Since
$g_{\varepsilon_j}\to g$ smoothly, the upper bound in
\eqref{eq:disk-approx-bilip} gives a uniform Lipschitz bound with respect to
$g$.  By the Arzel\`a--Ascoli theorem, after passing to a subsequence,
$T_{\varepsilon_j}$ converges uniformly to a map
$T:\overline{\B^2}\to\overline{\B^2}$.  Using
$g_{\varepsilon_j}\to g$ smoothly and $L_{\varepsilon_j}\leq1$, passing to
the limit in \eqref{eq:disk-approx-bilip} yields
\begin{equation}
 a|z-z'|\leq\dist_g(Tz,Tz')\leq|z-z'|,
 \quad \forall\,z,z'\in\overline{\B^2},\qquad
 T_\#\frac{\d x}{\omega_2}
 =\frac{\d\vol_g}{\vol_g(\B^2)}.
\end{equation}
The distance bounds make $T$ injective, and the measure identity and
compactness make it surjective.  Thus $T$ is a bi-Lipschitz homeomorphism, and
Lemma~\ref{lem:inverse-bound} gives the lower bound in
Theorem~\ref{thm:disk}.  In particular,
$T(\partial\overline{\B^2})=\partial\overline{\B^2}$.
Applying Milman's contraction principle with boundary
\cite[Theorem~3.6]{Milman} directly to $T$ gives the
stated Dirichlet and Neumann spectral inequalities.
\end{proof}

\section{Hemisphere contractions}\label{sec:hemisphere}
In this section we prove the hemisphere results using reflection symmetry in
the closed spherical inverse mean-curvature flow construction.  We first
treat uniform measures on geodesically convex subsets in every dimension.
In dimension two, an equivariant refinement yields contractions onto
Alexandrov disks of curvature at least one; applying this result to canonical
conformal metrics gives weighted contractions.

\subsection{Uniform targets}

Let $n\geq2$.  We identify $\Sph^n$ with the equator
$\Sph^{n+1}\cap\{s=0\}$ and let $\mathcal R(x,s)=(x,-s)$ be reflection across
this equator.  The closed construction has the following equivariant form.

\begin{lemma}\label{lem:hemisphere-equivariant-closed}
Let $\Sigma^n\subseteq\Sph^{n+1}$ be a smooth, closed, connected, embedded,
strictly convex hypersurface contained in an open hemisphere and invariant
under $\mathcal R$.  Assume that $\Sigma\cap\{s>0\}$ is connected and set
$\Sigma^+:=\Sigma\cap\{s\geq0\}$.  Then there exist an
$\mathcal R$-invariant round equator $E\subseteq\Sph^{n+1}$, a closed round
hemisphere $E^+\subseteq E$, and a bi-Lipschitz homeomorphism
\[
 P:E^+\longrightarrow\Sigma^+
\]
such that
\begin{equation}\label{eq:hemisphere-equivariant-half-measure}
	P_\#\frac{\d\vol_E|_{E^+}}{|E^+|}
	=\frac{\d\vol_\Sigma|_{\Sigma^+}}{|\Sigma^+|}.
\end{equation}
With $b_\Sigma:=|\Sigma^+|/|E^+|$,
\begin{equation}\label{eq:hemisphere-equivariant-half-distance}
 b_\Sigma\dist_E(z,z')
 \leq\dist_\Sigma(Pz,Pz')
 \leq\dist_E(z,z'),
 \qquad \forall\,z,z'\in E^+. 
\end{equation}
\end{lemma}

\begin{proof}
Run the inverse mean curvature flow \eqref{eq:spherical-imcf} from $\Sigma$ and
write $r=\mathcal R|_\Sigma$.  Uniqueness gives
$X_t\circ r=\mathcal R\circ X_t$.  By spherical polar duality
\cite[Section~4 and Theorem~1.1]{GerhardtSphere}, the pole $p$ of the
limiting equator $E=p^\perp\cap\Sph^{n+1}$ is the unique extinction point of
the associated polar flow.  This flow is also $\mathcal R$-invariant, so
$\mathcal R p=p$.  Hence $p\in\{s=0\}$, $e_s\in E$, and
$\mathcal R|_E$ is a nontrivial reflection.

Near the terminal time, $\Sigma_t$ is uniquely represented as a radial graph
over $E$.  The radial parametrizations $Q_t:E\to\Sigma_t$ in the proof of
Theorem~\ref{thm:extrinsic-unified}~(i) are therefore
$\mathcal R$-equivariant, as are $P_t=X_t^{-1}\circ Q_t$.  Applying the endpoint
argument in that proof yields an equivariant bi-Lipschitz homeomorphism
$P:E\to\Sigma$ with $\Lip(P)\leq1$ and
$P_\#(\d\vol_E/|E|)=\d\vol_\Sigma/|\Sigma|$.

Equivariance and injectivity show that $P$ maps the fixed equator of
$\mathcal R|_E$ onto $\Sigma\cap\{s=0\}$.  Since the fixed equator has zero
volume and $P$ is bi-Lipschitz, $\Sigma\cap\{s=0\}$ also has zero volume.
Since $\Sigma\cap\{s>0\}$ is connected, reflection symmetry implies that
$\Sigma\cap\{s<0\}$ is connected as well.  Hence these are precisely the two
components of $\Sigma\setminus\bigl(\Sigma\cap\{s=0\}\bigr)$. 
After interchanging the two hemispheres of $E$ if necessary, we therefore
have $P(E^+)=\Sigma^+$.  Reflection symmetry yields
$|E^+|=|E|/2$ and $|\Sigma^+|=|\Sigma|/2$,
and hence $b_\Sigma=|\Sigma|/|E|$.  Applying
Lemma~\ref{lem:inverse-bound} to the global map $P$ and then restricting to
$E^+$ gives \eqref{eq:hemisphere-equivariant-half-distance}.  Restricting
the global pushforward identity to the corresponding halves gives
\eqref{eq:hemisphere-equivariant-half-measure}. 
\end{proof}

We next approximate an arbitrary convex target by the upper halves of
strictly convex hypersurfaces invariant under reflection.  The construction
also gives the required convergence of volume measures and intrinsic
distances. 
\begin{lemma}\label{lem:hemisphere-thin-approximation}
Let $K\subseteq\Sph^n_+$ be closed and geodesically convex, with $|K|>0$.
There exist smooth, closed, connected, embedded, strictly convex hypersurfaces
$\Sigma_j\subseteq\Sph^{n+1}$, contained in an open hemisphere and invariant
under $\mathcal R$, such that, with
$\Sigma_j^+:=\Sigma_j\cap\{s\geq0\}$,
\begin{equation}\label{eq:hemisphere-thin-convergence}
 \dist_{\mathrm H}(\Sigma_j^+,K)\to0,
 \quad
 \d\vol_{\Sigma_j}|_{\Sigma_j^+}
 \rightharpoonup\mathbf 1_K\d\vol_{\mathrm{can}},
\end{equation}
as $j\to\infty$.  Here
$\dist_{\mathrm H}$
denotes Hausdorff distance in $\Sph^{n+1}$. Moreover, they may be chosen so
that whenever $p_j,q_j\in\Sigma_j^+$ converge respectively to $p,q\in K$
in $\Sph^{n+1}$, then
\begin{equation}\label{eq:hemisphere-intrinsic-distance-limit}
 \limsup_{j\to\infty}\dist_{\Sigma_j}(p_j,q_j)
 \leq\dist_{\mathrm{can}}(p,q).
\end{equation}
In addition, $\Sigma_j\cap\{s>0\}$ is connected for every $j$. 
\end{lemma}

\begin{proof}
\noindent\emph{Step 1: smooth targets in an open hemisphere.}
Let $N$ be the center of $\Sph^n_+$.  Since $|K|>0$, choose
$p\in\operatorname{int}(K)$ and set
$\delta_j:=(p\cdot N)/(j+1)$ and
$K_j:=K\cap\{x\cdot N\geq\delta_j\}$.
Then $K_j$ is contained in the open hemisphere centered at $N$ and
\begin{equation}
 \dist_{\mathrm H}(K_j,K)\to0,
 \qquad
 \mathbf 1_{K_j}\d\vol_{\mathrm{can}}\rightharpoonup
 \mathbf 1_K\d\vol_{\mathrm{can}}.
\end{equation}
Indeed, points of $K$ on the boundary equator are approached along the
geodesic segments from $p$, and that equator has zero $n$-dimensional volume.

Regard $N$ as $(N,0)\in\Sph^{n+1}$ and use the gnomonic coordinates 
\[
 \Gamma(Y)=\frac{Y}{Y\cdot N}-N,
 \qquad
 G(V)=\frac{N+V}{\sqrt{1+|V|^2}},
\]
on $\mathcal H_N:=\{Y\in \Sph^{n+1}:Y\cdot N>0\}$.  They commute with $\mathcal R$ and
map affine segments to reparametrized minimizing great-circle arcs.  They
therefore identify Euclidean convexity with geodesic convexity.  Thus
$A_j:=\Gamma(K_j)\subseteq\R^n$ is a compact convex body.

To approximate $A_j$ by smooth uniformly convex bodies, let $h_j$ be its
support function and set
$f_j(x):=\max_{u\in\Sph^{n-1}}(x\cdot u-h_j(u))$.
Then $f_j$ is convex and $1$-Lipschitz, since it is a supremum of affine
functions with unit slopes, and the separating-hyperplane theorem implies 
\[
 A_j=\{f_j\leq0\},
 \qquad
 \max\{f_j,0\}=\dist_{\mathrm{Euc}}(\,\cdot\,,A_j).
\]
For $\eta>0$, let $F_{j,\eta}:=f_j*\varphi_\eta+\eta|x|^2$,
where $\varphi_\eta$ is a standard mollifier.  Then
$D^2F_{j,\eta}\geq2\eta I$ and $F_{j,\eta}\to f_j$ locally uniformly.
Choose $R_j>0$ such that $A_j\subseteq\B_{R_j}(0)$ and set
$c_{j,\eta}:=2\eta(1+R_j^2)$ and
$A_{j,\eta}:=\{F_{j,\eta}\leq c_{j,\eta}\}$.  Since
$\norm{f_j*\varphi_\eta-f_j}_{C^0}\leq\eta$,
\[
 A_j\subseteq\{F_{j,\eta}<c_{j,\eta}\}
 \subseteq A_{j,\eta}
 \subseteq A_j+(c_{j,\eta}+\eta)\overline{\B^n}.
\]
Moreover, $F_{j,\eta}$ is strongly convex and coercive.  Its only critical
point is its global minimum, whose value is below $c_{j,\eta}$; hence
$c_{j,\eta}$ is a regular value.  Thus $A_{j,\eta}$ is a smooth uniformly
convex compact body and converges to $A_j$ in Hausdorff distance.

Set $S_{j,\eta}:=G(A_{j,\eta}\times\{0\})$.  The preceding inclusions, the
local Lipschitz continuity of $G$, and dominated convergence yield, for each
fixed $j$,
\[
 \dist_{\mathrm H}(S_{j,\eta},K_j)\to0,
 \qquad
 \mathbf 1_{S_{j,\eta}}\d\vol_{\mathrm{can}}
 \rightharpoonup\mathbf 1_{K_j}\d\vol_{\mathrm{can}}
\]
as $\eta\downarrow0$, where
$\bigl(G|_{\R^n\times\{0\}}\bigr)^*\d\vol_{\mathrm{can}}
=(1+|x|^2)^{-(n+1)/2}\d x$.  Using the bounded--Lipschitz metric on finite
measures on $\Sph^{n+1}$, choose $\eta_j\downarrow0$ so that both the Hausdorff and
bounded--Lipschitz errors above are at most $j^{-1}$.
Set $\rho_j:=F_{j,\eta_j}-c_{j,\eta_j}$,
$A'_j:=A_{j,\eta_j}$, and $S_j:=S_{j,\eta_j}$.  Then
\begin{equation}\label{eq:hemisphere-smoothed-target-limit}
 \dist_{\mathrm H}(S_j,K)\to0,
 \qquad
 \mathbf 1_{S_j}\d\vol_{\mathrm{can}}\rightharpoonup
 \mathbf 1_K\d\vol_{\mathrm{can}}.
\end{equation}

\smallskip
\noindent\emph{Step 2: symmetric strictly convex thickenings.}
For $\varepsilon>0$, let
\[
 B_{j,\varepsilon}:=
 \left\{(x,s):\rho_j(x)+\frac{s^2}{\varepsilon^2}\leq0\right\},
 \qquad
 \Sigma_{j,\varepsilon}:=G(\partial B_{j,\varepsilon}).
\]
The defining function has positive-definite Hessian.  Hence
$B_{j,\varepsilon}$ is compact and strictly convex.  Its zero level is
regular: a critical point would have $s=0$ and would be the unique minimum
of $\rho_j$, where $\rho_j<0$.  Thus
$\partial B_{j,\varepsilon}$ is smooth, closed, and strictly convex.  The
gnomonic map also preserves strict convexity.  Indeed, if $y\in
\partial B_{j,\varepsilon}$, $\zeta$ is the Euclidean unit normal there,
and $v,w$ are tangent vectors, then, for compatible choices of normal,
\begin{equation}
	 \operatorname{II}_{\Sph^{n+1}}(\d G_yv,\d G_yw)
	=\frac{\operatorname{II}_{\R^{n+1}}(v,w)}
	{\sqrt{1+|y|^2}\sqrt{1+(y\cdot\zeta)^2}}. 
\end{equation}
Therefore $\Sigma_{j,\varepsilon}\subseteq\mathcal H_N$ is a smooth, closed,
connected, embedded, strictly convex hypersurface invariant under
$\mathcal R$.  Its upper open half is connected
because it is the graph over $\operatorname{int}(A'_j)$ of
$s_{j,\varepsilon}(x):=\varepsilon\sqrt{-\rho_j(x)}$.

\smallskip
\noindent\emph{Step 3: Hausdorff and volume convergence.}
Parametrize the upper half by
$\Psi_{j,\varepsilon}(x)=G(x,s_{j,\varepsilon}(x))$.  Near
$\partial A'_j$, the defining-function property gives
$-\rho_j(x)\geq c_j\dist_{\mathrm{Euc}}(x,\partial A'_j)$ for some $c_j>0$.
Together with
$\nabla s_{j,\varepsilon}=-\varepsilon\nabla\rho_j/(2\sqrt{-\rho_j})$
and the boundedness of $\nabla\rho_j$, this gives, after enlarging $C_j$ away
from the boundary,
\begin{equation}\label{eq:hemisphere-graph-gradient}
 |\nabla s_{j,\varepsilon}(x)|
 \leq \frac{C_j\varepsilon}
 {\sqrt{\dist_{\mathrm{Euc}}(x,\partial A'_j)}}.
\end{equation}
Thus $s_{j,\varepsilon}\to0$ uniformly and
$\Sigma_{j,\varepsilon}^+\to S_j$ in Hausdorff distance.  The graph
Jacobian is
\begin{equation}
	 J_{j,\varepsilon}(x)=
	\frac{\sqrt{1+|\nabla s_{j,\varepsilon}|^2+
			(s_{j,\varepsilon}-x\cdot\nabla s_{j,\varepsilon})^2}}
	{(1+|x|^2+s_{j,\varepsilon}^2)^{(n+1)/2}}. 
\end{equation}
It converges pointwise in $\operatorname{int}(A'_j)$ to
$(1+|x|^2)^{-(n+1)/2}$.  Estimate
\eqref{eq:hemisphere-graph-gradient} bounds it by
$C_j(1+\dist_{\mathrm{Euc}}(x,\partial A'_j)^{-1/2})$, which is integrable
in tubular coordinates.  Apply the smooth area formula first on
$A'_{j,\delta}:=
\{x\in A'_j:\dist_{\mathrm{Euc}}(x,\partial A'_j)\geq\delta\}$,
and then let $\delta\downarrow0$.  These sets exhaust
$\operatorname{int}(A'_j)$, while the omitted equatorial set is a smooth
$(n-1)$-dimensional submanifold and hence has zero $n$-volume.  The preceding
integrable bound and dominated convergence therefore give, for every
continuous $f$,
\begin{equation}
 \int_{\Sigma_{j,\varepsilon}^+}f\,\d\vol_{\Sigma_{j,\varepsilon}}
 \to \int_{S_j}f\,\d\vol_{\mathrm{can}}.
\end{equation}

\smallskip
\noindent\emph{Step 4: intrinsic distances and the diagonal sequence.}
Define the vertical projection by
$\pi_{j,\varepsilon}(\Psi_{j,\varepsilon}(x)):=G(x,0)$.
For $x,y\in A'_j$, let $\ell$ be the Euclidean segment joining them and set
$v:=s_{j,\varepsilon}\circ\ell$.  Since $-\rho_j$ is nonnegative and
concave and the square-root function is increasing and concave, $v$ is
nonnegative and concave.  Thus $v$ is absolutely continuous and
$\operatorname{Var}(v)\leq2\norm{s_{j,\varepsilon}}_{C^0(A'_j)}$.
On a fixed compact neighborhood of $A'_j\times\{0\}$, the differential of
$G$ is bounded and Lipschitz.  Hence, almost everywhere along $\ell$,
\begin{equation}
	\left|\d G_{(\ell,v)}(\ell',v')\right|
	\leq \left|\d G_{(\ell,0)}(\ell',0)\right|
	+C_j\bigl(|v|\,|\ell'|+|v'|\bigr). 
\end{equation}
Integrating over compact subsets of $\ell$ and then letting these
subsets exhaust $\ell$ yields
\begin{equation}
	 \operatorname{Length}\bigl(G(\ell,v)\bigr)
	\leq \operatorname{Length}\bigl(G(\ell,0)\bigr)
	+C_j\left(
	\norm{s_{j,\varepsilon}}_{C^0}\operatorname{Length}_{\mathrm{Euc}}(\ell)
	+\operatorname{Var}(v)\right). 
\end{equation}
Since $A'_j$ is bounded and
$\norm{s_{j,\varepsilon}}_{C^0}\to0$, the error is bounded by a quantity
$\omega_j(\varepsilon)\to0$, uniformly in $x,y$.  Moreover,
$G(\ell,0)$ is the minimizing short great-circle arc joining its endpoints.
Hence
\begin{equation}\label{eq:hemisphere-upper-sheet-distance}
 \dist_{\Sigma_{j,\varepsilon}}(p,q)
 \leq \dist_{\mathrm{can}}
 \bigl(\pi_{j,\varepsilon}(p),\pi_{j,\varepsilon}(q)\bigr)+
 \omega_j(\varepsilon),
 \qquad p,q\in\Sigma_{j,\varepsilon}^+.
\end{equation}

The uniform convergence of $s_{j,\varepsilon}$ also shows that the maximal
distance between a point of $\Sigma_{j,\varepsilon}^+$ and its vertical
projection tends to zero.  Choose $\varepsilon_j\downarrow0$ so that the
Hausdorff, bounded--Lipschitz, length and vertical-projection errors above
are all at most $j^{-1}$.
Denote 
$\Sigma_j:=\Sigma_{j,\varepsilon_j}$.  Together with
\eqref{eq:hemisphere-smoothed-target-limit}, this proves
\eqref{eq:hemisphere-thin-convergence}.  Finally, if
$p_j,q_j\in\Sigma_j^+$ converge to $p,q\in K$, respectively, then their
vertical projections converge to $p,q$.  Hence
\eqref{eq:hemisphere-upper-sheet-distance} yields 
\begin{equation}
	 \limsup_{j\to\infty}\dist_{\Sigma_j}(p_j,q_j)
	\leq\lim_{j\to\infty}\dist_{\mathrm{can}}
	\bigl(\pi_{j,\varepsilon_j}(p_j),
	\pi_{j,\varepsilon_j}(q_j)\bigr)
	=\dist_{\mathrm{can}}(p,q), 
\end{equation}
which proves \eqref{eq:hemisphere-intrinsic-distance-limit}.
\end{proof}

\begin{proof}[Proof of Theorem~\ref{thm:hemisphere-transport}]
\noindent\emph{Step 1: prelimit maps and compactness.}
Apply Lemma~\ref{lem:hemisphere-equivariant-closed} to the hypersurfaces in
Lemma~\ref{lem:hemisphere-thin-approximation}.  After identifying the
source halves with $\Sph^n_+$, we obtain bi-Lipschitz homeomorphisms
$T_j:\Sph^n_+\to\Sigma_j^+$ satisfying
\begin{equation}\label{eq:hemisphere-prelimit-estimates}
 \dist_{\mathrm{can}}(T_jz,T_jz')
 \leq\dist_{\Sigma_j}(T_jz,T_jz')
 \leq\dist_{\mathrm{can}}(z,z'),
 \qquad
 (T_j)_\#\sigma_+
 =\frac{\d\vol_{\Sigma_j}|_{\Sigma_j^+}}{|\Sigma_j^+|}.
\end{equation}
Let $P_j:E_j\to\Sigma_j$ denote the global equivariant endpoint map before
restriction.  Reflection symmetry gives
$|\Sigma_j|/|E_j|=|\Sigma_j^+|/|\Sph^n_+|=:b_j$.  Applying
Lemma~\ref{lem:inverse-bound} to $P_j$, rather than to its restriction, and
using $\Lip(P_j)\leq1$ yields 
\begin{equation}\label{eq:hemisphere-prelimit-lower}
 b_j\dist_{\mathrm{can}}(z,z')
 \leq\dist_{\Sigma_j}(T_jz,T_jz').
\end{equation}
By the Arzel\`a--Ascoli theorem and
\eqref{eq:hemisphere-thin-convergence}, after passing to a subsequence,
$T_j$ converges uniformly to a map $T:\Sph^n_+\to K$.  Passing to the limit
in the upper estimate in \eqref{eq:hemisphere-prelimit-estimates} shows that
$T$ is $1$-Lipschitz.

\smallskip
\noindent\emph{Step 2: the limiting measure and surjectivity.}
The weak convergence in
\eqref{eq:hemisphere-thin-convergence}, first tested against $1$, yields 
$|\Sigma_j^+|\to|K|$.  Hence, for every continuous $f$ on $\Sph^{n+1}$,
\begin{equation}
	 \int_{\Sph^n_+}f(Tx)\,\d\sigma_+(x)
	=\lim_{j\to\infty}\frac1{|\Sigma_j^+|}
	\int_{\Sigma_j^+}f\,\d\vol_{\Sigma_j}
	=\frac1{|K|}\int_K f\,\d\vol_{\mathrm{can}}. 
\end{equation}
This proves the pushforward identity in
Theorem~\ref{thm:hemisphere-transport}.  Its right-hand side has
support $K$, while the image of $T$ is compact and contained in $K$;
therefore $T$ is surjective.

\smallskip
\noindent\emph{Step 3: the lower Lipschitz bound.}
We have $b_j\to b_K$.  For fixed $z,z'$, apply
\eqref{eq:hemisphere-intrinsic-distance-limit} to $T_jz,T_jz'$ and combine it
with \eqref{eq:hemisphere-prelimit-lower} to obtain
$b_K\dist_{\mathrm{can}}(z,z')\leq\dist_{\mathrm{can}}(Tz,Tz')$.
Thus $T$ is injective with $b_K^{-1}$-Lipschitz inverse.  Together with
Step~2, this shows that $T$ is a bi-Lipschitz homeomorphism.  Applying
Lemma~\ref{lem:inverse-bound} to $T:\Sph^n_+\to K$, with
$L=\Lip(T)$, proves \eqref{eq:hemisphere-contraction} and the theorem.
\end{proof}

\subsection{Alexandrov disks and weighted targets}

Let $r_0:\Sph^2\to\Sph^2$ denote reflection across the boundary of
$\Sph^2_+$.  We first establish the following equivariant refinement of
Corollary~\ref{cor:alexandrov}, which will be applied to metric doubles.

\begin{proposition}[Equivariant contractions onto Alexandrov two-spheres]
\label{prop:hemisphere-equivariant-alexandrov}
Let $(X,\dist_X)$ be an Alexandrov surface of curvature at least one and
homeomorphic to $\Sph^2$.  Suppose that $X$ admits an isometric involution
$\iota$ whose fixed-point set is a simple closed curve, and assume that
$\iota$ exchanges the two components of the complement of this curve.  Then
there exists a bi-Lipschitz homeomorphism
\[
 F:(\Sph^2,g_{\mathrm{can}})\longrightarrow(X,\dist_X)
\]
such that
\begin{equation}\label{eq:equivariant-alexandrov-measure}
 F\circ r_0=\iota\circ F,
 \qquad
 F_\#\frac{\d\vol_{\mathrm{can}}}{|\Sph^2|}
 =\frac{\Hh_X^2}{\Hh^2(X)},
 \qquad
 L:=\Lip(F)\leq1.
\end{equation}
With $a_X:=\Hh^2(X)/|\Sph^2|$,
\begin{equation}\label{eq:equivariant-alexandrov-bilip}
 \frac{a_X}{L}\dist_{\mathrm{can}}(z,z')
 \leq\dist_X(Fz,Fz')
 \leq L\dist_{\mathrm{can}}(z,z'),
 \qquad \forall\,z,z'\in\Sph^2.
\end{equation}
\end{proposition}

\begin{proof}
\noindent\emph{Step 1: equivariant smooth approximation.}
Since $\iota$ exchanges the two sides of its fixed curve, it reverses
orientation.  By conformal uniformization
\cite[Theorem~7.6]{Troyanov}, followed by a M\"obius change of coordinates, there exists a conformal homeomorphism $\Phi:\Sph^2\to X$ such that
$\Phi\circ r_0=\iota\circ\Phi$.
Indeed, the induced map on $\Sph^2$ is an antiholomorphic involution, and
sending three points of its fixed circle to $0,1,\infty$ identifies it with
complex conjugation.  Set
$\dist_X^\Phi(p,q):=\dist_X(\Phi p,\Phi q)$ and
$\mu_X^\Phi:=(\Phi^{-1})_\#\Hh_X^2$.
In the conformal-potential sense used in
\cite[Notation~11.1]{LinWangXu}, write the pulled-back metric as
$e^{2u}g_{\mathrm{can}}$.  Since $\iota$ is an isometry and $\Phi$ is
equivariant, $u\circ r_0=u$ almost everywhere.  The round heat semigroup
commutes with $r_0$, so the regularizations
$g_j=e^{2P_{t_j}u}g_{\mathrm{can}}$, $t_j\downarrow0$, are
$r_0$-invariant.  By \cite[Lemmas~11.7 and~11.8]{LinWangXu}, we have
\begin{equation}\label{eq:equivariant-alexandrov-approximation}
 K_{g_j}\geq1,
 \qquad
 \varepsilon_j:=
 \norm{\dist_{g_j}-\dist_X^\Phi}_{C^0(\Sph^2\times\Sph^2)}
 \to0,
 \qquad
 \norm{\d\vol_{g_j}-\mu_X^\Phi}_{\mathrm{TV}}\to0.
\end{equation}
Set $V_j:=\vol_{g_j}(\Sph^2)$ and $a_j:=V_j/|\Sph^2|$.  Then
$V_j\to\Hh^2(X)$ and $a_j\to a_X$.

\smallskip
\noindent\emph{Step 2: smooth equivariant contractions.}
Fix $j$ and set $g_{j,\varepsilon}:=(1-\varepsilon)g_j$ for
$0<\varepsilon<1$.  Then $K_{g_{j,\varepsilon}}>1$, and by the spherical Weyl
theorem \cite[Theorem~1.3]{Lu} and
Lemma~\ref{lem:automatic-regularity}, there exists a smooth strictly convex
isometric embedding
$X_{j,\varepsilon}:(\Sph^2,g_{j,\varepsilon})\to\Sph^3$, whose image lies in
an open hemisphere by \cite[Theorem~1.1]{doCarmoWarner}.
Let $\pi:\Sph^3\to\mathbb E^3=\Sph^3/\{\pm1\}$ be the elliptic quotient.
Projecting to elliptic space and applying global rigidity
\cite[Chapter~VI, Section~5, Theorem~1]{PogorelovElliptic} to
$\pi\circ X_{j,\varepsilon}$ and
$\pi\circ X_{j,\varepsilon}\circ r_0$ then yields a lift
$A_{j,\varepsilon}\in O(4)$.  The two spherical lifts differ by a constant
sign, which may be absorbed into $A_{j,\varepsilon}$, so
$X_{j,\varepsilon}\circ r_0
=A_{j,\varepsilon}\circ X_{j,\varepsilon}$.
Applying this identity twice shows that $A_{j,\varepsilon}^2$ is the
identity on the embedded surface.  This surface spans $\R^4$.  Otherwise, it
would lie in a great sphere of dimension at most two, and invariance of
domain would force it to be a great two-sphere, contrary to strict
convexity.  Hence $A_{j,\varepsilon}^2=\operatorname{Id}$.  Let
$V^+_{j,\varepsilon}$ be the $+1$ eigenspace of $A_{j,\varepsilon}$.  It
contains the image of the equator fixed by $r_0$, so
$\dim V^+_{j,\varepsilon}\geq2$.  Moreover, $A_{j,\varepsilon}$ is
nontrivial by injectivity of $X_{j,\varepsilon}$.  If
$\dim V^+_{j,\varepsilon}=2$, that image would be the great circle
$V^+_{j,\varepsilon}\cap\Sph^3$ and would have zero normal curvature,
contrary to strict convexity.  Thus $\dim V^+_{j,\varepsilon}=3$, and
$A_{j,\varepsilon}$ is reflection across the great two-sphere
$V^+_{j,\varepsilon}\cap\Sph^3$.

Put $\Sigma_{j,\varepsilon}:=X_{j,\varepsilon}(\Sph^2)$.  The polar-flow
argument and the global construction in the proof of
Lemma~\ref{lem:hemisphere-equivariant-closed} give an
$A_{j,\varepsilon}$-invariant round equator $E_{j,\varepsilon}$, on which
$A_{j,\varepsilon}$ is a nontrivial reflection, and an equivariant
bi-Lipschitz homeomorphism
$P_{j,\varepsilon}:E_{j,\varepsilon}\to\Sigma_{j,\varepsilon}$ preserving
normalized volume.  Choose a round isometry
$J_{j,\varepsilon}:\Sph^2\to E_{j,\varepsilon}$ which intertwines $r_0$
with $A_{j,\varepsilon}|_{E_{j,\varepsilon}}$, and set
\[
 S_{j,\varepsilon}:=X_{j,\varepsilon}^{-1}
 \circ P_{j,\varepsilon}\circ J_{j,\varepsilon}.
\]
By construction, $S_{j,\varepsilon}$ is equivariant and preserves
normalized volume:
\begin{equation}
	S_{j,\varepsilon}\circ r_0=r_0\circ S_{j,\varepsilon},
	\qquad
	(S_{j,\varepsilon})_\#
	\frac{\d\vol_{\mathrm{can}}}{|\Sph^2|}
	=\frac{\d\vol_{g_j}}{V_j}. 
\end{equation}
Since the volume ratio in the global construction is
$(1-\varepsilon)a_j$ and
$\dist_{g_{j,\varepsilon}}=\sqrt{1-\varepsilon}\,\dist_{g_j}$, the global
distance estimate yields 
\begin{equation}
	 \sqrt{1-\varepsilon}\,a_j\dist_{\mathrm{can}}(z,z')
	\leq\dist_{g_j}(S_{j,\varepsilon}z,S_{j,\varepsilon}z')
	\leq(1-\varepsilon)^{-1/2}\dist_{\mathrm{can}}(z,z'). 
\end{equation}
Letting $\varepsilon\downarrow0$ and applying the compactness argument from
Step~2 of the proof of Theorem~\ref{thm:milman}, we obtain an
$r_0$-equivariant bi-Lipschitz homeomorphism
$S_j:\Sph^2\to(\Sph^2,g_j)$ such that
\begin{equation}\label{eq:equivariant-smooth-maps}
 S_j\circ r_0=r_0\circ S_j,
 \quad
 a_j\dist_{\mathrm{can}}(z,z')
 \leq\dist_{g_j}(S_jz,S_jz')
 \leq\dist_{\mathrm{can}}(z,z'),
 \quad
 (S_j)_\#\frac{\d\vol_{\mathrm{can}}}{|\Sph^2|}
 =\frac{\d\vol_{g_j}}{V_j}.
\end{equation}

\smallskip
\noindent\emph{Step 3: passage to the Alexandrov limit.}
Set $F_j:=\Phi\circ S_j$.  The compactness argument in Step~2 of the proof of
Corollary~\ref{cor:alexandrov} applies without change.  Moreover,
$F_j\circ r_0=\iota\circ F_j$, so the uniform limit is equivariant.  Thus there exists a
bi-Lipschitz homeomorphism $F:\Sph^2\to X$ such that 
\begin{equation}
	 F\circ r_0=\iota\circ F,
	\qquad
	F_\#\frac{\d\vol_{\mathrm{can}}}{|\Sph^2|}
	=\frac{\Hh_X^2}{\Hh^2(X)},
	\qquad
	a_X\dist_{\mathrm{can}}(z,z')
	\leq\dist_X(Fz,Fz')
	\leq\dist_{\mathrm{can}}(z,z'). 
\end{equation}
Applying Lemma~\ref{lem:inverse-bound} to $F$ yields
\eqref{eq:equivariant-alexandrov-bilip}.
\end{proof}

\begin{proof}[Proof of Theorem~\ref{thm:alexandrov-disk}]
Let $X=DY$ be the metric double of $Y$ along its boundary.  The Alexandrov
doubling theorem \cite[Theorem~2.1]{Petrunin} shows that $X$ is an Alexandrov
two-sphere of curvature at least one.  Let $\iota$ be the involution
exchanging the two copies of $Y$.
Proposition~\ref{prop:hemisphere-equivariant-alexandrov} gives an equivariant map
$F:\Sph^2\to X$.  Equivariance and injectivity identify the fixed sets and
the two component closures.  After replacing $F$ by $\iota\circ F$ if
necessary, set $T:=F|_{\Sph^2_+}:\Sph^2_+\to Y$.

Set $\tilde L:=\Lip(F)$ and $L:=\Lip(T)$. 
Equivariance shows that the restriction of $F$ to the other hemisphere also
has Lipschitz constant $L$.  For $z,z'$ in different hemispheres, let $q$ be
an intersection of a minimizing geodesic from $z$ to $z'$ with the equator.  Then 
\begin{equation}
	 \dist_X(Fz,Fz')
	\leq L\bigl(\dist_{\mathrm{can}}(z,q)+\dist_{\mathrm{can}}(q,z')\bigr)
	=L\dist_{\mathrm{can}}(z,z'). 
\end{equation}
Together with the same-hemisphere estimate, this yields 
$\tilde L=L\leq1$.

The inclusion and the $1$-Lipschitz folding map $X\to Y$ show that each copy
of $Y$ is isometrically embedded in $X$.  Since $\Hh^2(\partial Y)=0$,
\begin{equation}
	 \Hh^2(X)=2\Hh^2(Y),
	\qquad
	a_X=\frac{\Hh^2(X)}{|\Sph^2|}
	=\frac{\Hh^2(Y)}{|\Sph^2_+|}=b_Y. 
\end{equation}
Restricting \eqref{eq:equivariant-alexandrov-measure} and
\eqref{eq:equivariant-alexandrov-bilip} to the corresponding halves yields 
\begin{equation}
	\frac{b_Y}{L}\dist_{\mathrm{can}}(z,z')
	\leq\dist_Y(Tz,Tz')
	\leq L\dist_{\mathrm{can}}(z,z'), \qquad T_\#\sigma_+=\frac{\Hh_Y^2}{\Hh^2(Y)},
\end{equation}
which proves Theorem~\ref{thm:alexandrov-disk}.
\end{proof}

The weighted corollary follows directly from
Theorem~\ref{thm:alexandrov-disk} applied to the canonical conformal metric.

\begin{proof}[Proof of Corollary~\ref{cor:weighted-hemisphere}]
Set $\hat g:=e^{-V}g_0$.  The conformal transformation laws yield 
\begin{equation}
 K_{\hat g}
 =e^V\left(1+\frac12\Delta_0V\right),
 \qquad
 \hat\kappa
 =e^{V/2}\left(\kappa_0-\frac12\partial_\eta V\right).
\end{equation}
Thus \eqref{eq:intro-weighted-conditions} is equivalent to
$K_{\hat g}\geq1$ and local convexity of $\partial\Omega$ in
$(\Omega,\hat g)$.  Hence $(\Omega,\dist_{\hat g})$ is an Alexandrov
disk of curvature at least one
\cite[Section~2.9, Example~(1)]{BuragoGromovPerelman}.  Since
$\d\vol_{\hat g}=e^{-V}\d\vol_0$ and
$\vol_{\hat g}(\Omega)=Z_V$, Theorem~\ref{thm:alexandrov-disk} gives a
bi-Lipschitz homeomorphism $T:\Sph^2_+\to\Omega$ such that
\begin{equation}
	\frac{Z_V}{2\pi}\dist_{\mathrm{can}}(z,z')
	\leq\dist_{\hat g}(Tz,Tz')
	\leq\dist_{\mathrm{can}}(z,z'), \qquad  T_\#\sigma_+
	=\frac{\d\vol_{\hat g}}{Z_V}
	=\frac{e^{-V}}{Z_V}\d\vol_0. 
\end{equation}
Since $m_V\leq V\leq0$, for all $x,y\in\Omega$,
\begin{equation}
 e^{m_V/2}\dist_{\hat g}(x,y)
 \leq\dist_{0,\Omega}(x,y)
 \leq\dist_{\hat g}(x,y).    
\end{equation}
Combining this estimate with the preceding display gives
\eqref{eq:intro-weighted-bilip} and completes the proof.
\end{proof}

\bigskip
\noindent\textbf{Acknowledgments.}
The authors thank Guoyi Xu for his interest in this work and for valuable
suggestions.  S. Aryan thanks his advisor, Tobias Colding, for his support and
for many insightful conversations, and Tang-Kai Lee for helpful conversations
about geometric flows.

\medskip
\noindent\textbf{Declaration.}
The authors declare no competing interests.  No datasets were generated or
analyzed in this study.

\medskip
\noindent\textbf{Funding.}
B.-X. Han and Z.-N. Zhu were supported in part by the National Key R\&D Program
of China (2021YFA1000900, 2021YFA1002200), the Shandong Provincial Natural
Science Foundation (ZR2025QB05), and the Taishan Scholars Program of Shandong
Province (tsqn202408059).  S.~Aryan was supported in part by NSF Grant
DMS-2405393 and a Simons Dissertation Fellowship.


\begin{thebibliography}{99}

\bibitem{Andrews}
B.~Andrews,
\emph{Pinching estimates and motion of hypersurfaces by curvature functions},
J. Reine Angew. Math. \textbf{608} (2007), 17--33.

\bibitem{Aryan}
S.~Aryan,
\emph{Spectral obstructions to contracting transport maps on curved spaces},
arXiv:2605.24705v1 (2026), accessed July 28, 2026.

\bibitem{BakriCritical}
L.~Bakri,
\emph{Critical sets of eigenfunctions of the Laplacian},
J. Geom. Phys. \textbf{62} (2012), no.~10, 2024--2037,
doi:10.1016/j.geomphys.2012.05.006.

\bibitem{BeckJerison}
T.~Beck and D.~Jerison,
\emph{The Friedland--Hayman inequality and Caffarelli's contraction theorem},
J. Math. Phys. \textbf{62} (2021), no.~10, Article No.~101504,
doi:10.1063/5.0046058.

\bibitem{Bertrand}
J.~Bertrand,
\emph{Existence and uniqueness of optimal maps on Alexandrov spaces},
Adv. Math. \textbf{219} (2008), no.~3, 838--851.

\bibitem{BuragoBuragoIvanov}
D.~Burago, Y.~Burago, and S.~Ivanov,
\emph{A Course in Metric Geometry},
Graduate Studies in Mathematics, vol.~33,
American Mathematical Society, Providence, RI, 2001.

\bibitem{BuragoGromovPerelman}
Yu.~D. Burago, M.~L. Gromov, and G.~Ya. Perel'man,
\emph{A.~D. Alexandrov spaces with curvature bounded below},
Russian Math. Surveys \textbf{47} (1992), no.~2, 1--58,
doi:10.1070/RM1992v047n02ABEH000877.

\bibitem{Caffarelli}
L.~A. Caffarelli,
\emph{Monotonicity properties of optimal transportation and the FKG and
related inequalities},
Comm. Math. Phys. \textbf{214} (2000), no.~3, 547--563.

\bibitem{carlier2024optimal}
G.~Carlier, A.~Figalli, and F.~Santambrogio,
\emph{On optimal transport maps between $1/d$-concave densities},
Ann. Inst. H. Poincar\'e Anal. Non Lin\'eaire \textbf{43} (2026), no.~2,
483--500,
doi:10.4171/AIHPC/148.

\bibitem{ChenLi}
G.-Q.~G. Chen and S.~Li,
\emph{Global weak rigidity of the Gauss--Codazzi--Ricci equations and
isometric immersions of Riemannian manifolds with lower regularity},
J. Geom. Anal. \textbf{28} (2018), no.~3, 1957--2007,
doi:10.1007/s12220-017-9893-1.

\bibitem{chewi2023entropic}
S.~Chewi and A.-A.~Pooladian,
\emph{An entropic generalization of Caffarelli's contraction theorem via
covariance inequalities},
C. R. Math. Acad. Sci. Paris \textbf{361} (2023), no.~G9, 1471--1482,
doi:10.5802/crmath.486.

\bibitem{doCarmoWarner}
M.~P. do Carmo and F.~W. Warner,
\emph{Rigidity and convexity of hypersurfaces in spheres},
J. Differential Geom. \textbf{4} (1970), 133--144.

\bibitem{CM}
T.~H.~Colding and W.~P.~Minicozzi,
\emph{Weyl type bounds for harmonic functions},
Inventiones mathematicae. \textbf{131} (1998), no.~2, 257--298.


\bibitem{FathiFradeliziGozlanZugmeyer}
M.~Fathi, M.~Fradelizi, N.~Gozlan, and S.~Zugmeyer,
\emph{Some obstructions to contraction theorems on the half-sphere},
in \emph{High Dimensional Probability X}, Progress in Probability, vol.~82,
Birkh\"auser, Cham, 2026, 149--159,
doi:10.1007/978-3-032-06057-0\_6.

\bibitem{fathi2020proof}
M.~Fathi, N.~Gozlan, and M.~Prod'homme,
\emph{A proof of the Caffarelli contraction theorem via entropic
regularization},
Calc. Var. Partial Differential Equations \textbf{59} (2020), no.~3,
Article No.~96, 18 pp.,
doi:10.1007/s00526-020-01754-0.

\bibitem{fathi2024transportation}
M.~Fathi, D.~Mikulincer, and Y.~Shenfeld,
\emph{Transportation onto log-Lipschitz perturbations},
Calc. Var. Partial Differential Equations \textbf{63} (2024), no.~3,
Article No.~61, 25 pp.,
doi:10.1007/s00526-023-02652-x.

\bibitem{feyel2004monge}
D.~Feyel and A.~S. \"Ust\"unel,
\emph{Monge--Kantorovitch measure transportation and Monge--Amp\`ere equation
on Wiener space},
Probab. Theory Related Fields \textbf{128} (2004), no.~3, 347--385,
doi:10.1007/s00440-003-0307-x.

\bibitem{GeSerres}
Y.~Ge and J.~Serres,
\emph{A generalization of Caffarelli's contraction theorem to nearly
spherical manifolds},
arXiv:2512.01496v1 (2025), accessed July 28, 2026.

\bibitem{GerhardtSphere}
C.~Gerhardt,
\emph{Curvature flows in the sphere},
J. Differential Geom. \textbf{100} (2015), no.~2, 301--347.

\bibitem{GerhardtCurvatureProblems}
C.~Gerhardt,
\emph{Curvature Problems},
Series in Geometry and Topology, vol.~39,
International Press, Somerville, MA, 2006.

\bibitem{GigliHan}
N.~Gigli and B.-X.~Han,
\emph{Independence on $p$ of weak upper gradients on RCD spaces},
J. Funct. Anal. \textbf{271} (2016), no.~1, 1--11,
doi:10.1016/j.jfa.2016.04.014.

\bibitem{GimenoGonzalez}
V.~Gimeno i Garcia and F.~Gonz\'alez-Ib\'a\~nez,
\emph{Evolution of the torsional rigidity under geometric flows},
Potential Anal. \textbf{64} (2026), Article No.~29,
doi:10.1007/s11118-025-10243-y.

\bibitem{Gromov}
M.~Gromov,
\emph{Metric Structures for Riemannian and Non-Riemannian Spaces},
Modern Birkh\"auser Classics,
Birkh\"auser Boston, Boston, MA, 2007,
doi:10.1007/978-0-8176-4583-0.

\bibitem{HoPyo}
P.~T. Ho and J.~Pyo,
\emph{First eigenvalues of free boundary hypersurfaces in the unit ball along
the inverse mean curvature flow},
Differential Geom. Appl. \textbf{93} (2024), Article No.~102095,
doi:10.1016/j.difgeo.2023.102095.

\bibitem{KarpukhinNadirashviliPenskoiPolterovich}
M.~Karpukhin, N.~Nadirashvili, A.~V. Penskoi, and I.~Polterovich,
\emph{An isoperimetric inequality for Laplace eigenvalues on the sphere},
J. Differential Geom. \textbf{118} (2021), no.~2, 313--333,
doi:10.4310/jdg/1622743142.

\bibitem{KimMilman}
Y.-H.~Kim and E.~Milman,
\emph{A generalization of Caffarelli's contraction theorem via (reverse) heat
flow},
Math. Ann. \textbf{354} (2012), no.~3, 827--862,
doi:10.1007/s00208-011-0749-x.

\bibitem{Kirchheim}
B.~Kirchheim,
\emph{Rectifiable metric spaces: local structure and regularity of the
Hausdorff measure},
Proc. Amer. Math. Soc. \textbf{121} (1994), no.~1, 113--123.

\bibitem{Koerber}
T.~K\"orber,
\emph{A free boundary isometric embedding problem in the unit ball},
Calc. Var. Partial Differential Equations \textbf{61} (2022), no.~2,
Article No.~50, 36 pp.,
doi:10.1007/s00526-021-02173-5.

\bibitem{KupfermanMaorShachar}
R.~Kupferman, C.~Maor, and A.~Shachar,
\emph{Reshetnyak rigidity for Riemannian manifolds},
Arch. Ration. Mech. Anal. \textbf{231} (2019), no.~1, 367--408,
doi:10.1007/s00205-018-1282-9.

\bibitem{KuwaeMachigashiraShioya}
K.~Kuwae, Y.~Machigashira, and T.~Shioya,
\emph{Sobolev spaces, Laplacian, and heat kernel on Alexandrov spaces},
Math. Z. \textbf{238} (2001), no.~2, 269--316,
doi:10.1007/s002090100252.

\bibitem{LambertScheuer}
B.~Lambert and J.~Scheuer,
\emph{The inverse mean curvature flow perpendicular to the sphere},
Math. Ann. \textbf{364} (2016), no.~3--4, 1069--1093,
doi:10.1007/s00208-015-1248-2.

\bibitem{LambertScheuerInequality}
B.~Lambert and J.~Scheuer,
\emph{A geometric inequality for convex free boundary hypersurfaces in the
unit ball},
Proc. Amer. Math. Soc. \textbf{145} (2017), no.~9, 4009--4020,
doi:10.1090/proc/13516.

\bibitem{LauretBerger}
E.~A. Lauret,
\emph{The smallest Laplace eigenvalue of homogeneous $3$-spheres},
Bull. Lond. Math. Soc. \textbf{51} (2019), no.~1, 49--69,
doi:10.1112/blms.12213.

\bibitem{LinWangXu}
S.~Lin, H.~Wang, and G.~Xu,
\emph{Eigenvalues on spheres},
arXiv:2607.11544v1 (2026), accessed July 28, 2026.

\bibitem{lopez2025bakry}
P.~L\'opez-Rivera,
\emph{A Bakry--\'Emery approach to Lipschitz transportation on manifolds},
Potential Anal. \textbf{62} (2025), no.~2, 331--353,
doi:10.1007/s11118-024-10138-4.

\bibitem{Lu}
S.~Lu,
\emph{On Weyl's embedding problem in Riemannian manifolds},
Int. Math. Res. Not. IMRN \textbf{2020} (2020), no.~11, 3229--3259.

\bibitem{MakowskiScheuer}
M.~Makowski and J.~Scheuer,
\emph{Rigidity results, inverse curvature flows, and Alexandrov--Fenchel type
inequalities in the sphere},
Asian J. Math. \textbf{20} (2016), no.~5, 869--892.

\bibitem{MarquardtThesis}
T.~Marquardt,
\emph{The inverse mean curvature flow for hypersurfaces with boundary},
Ph.D. thesis, Freie Universit\"at Berlin, 2012,
doi:10.17169/refubium-12953.

\bibitem{mikulincer2024brownian}
D.~Mikulincer and Y.~Shenfeld,
\emph{The Brownian transport map},
Probab. Theory Related Fields \textbf{190} (2024), 379--444,
doi:10.1007/s00440-024-01286-0.

\bibitem{Milman}
E.~Milman,
\emph{Spectral estimates, contractions and hypercontractivity},
J. Spectr. Theory \textbf{8} (2018), no.~2, 669--714,
doi:10.4171/JST/210.

\bibitem{Petrunin}
A.~Petrunin,
\emph{Applications of quasigeodesics and gradient curves},
in \emph{Comparison Geometry}, Math. Sci. Res. Inst. Publ., vol.~30,
Cambridge University Press, Cambridge, 1997, 203--219.

\bibitem{PogorelovElliptic}
A.~V. Pogorelov,
\emph{Topics in the Theory of Surfaces in Elliptic Space},
Gordon and Breach, New York, 1961.

\bibitem{Serres}
J.~Serres,
\emph{Contractive transport maps from $\mathbb S^2$ to nearly spherical
surfaces with positive Ricci curvature},
Nonlinear Anal. \textbf{267} (2026), Article No.~114058,
doi:10.1016/j.na.2026.114058.

\bibitem{Troyanov}
M.~Troyanov,
\emph{On Alexandrov's surfaces with bounded integral curvature},
in \emph{Reshetnyak's Theory of Subharmonic Metrics},
Springer, Cham, 2023, 9--34,
doi:10.1007/978-3-031-24255-7\_2.

\bibitem{Villani}
C.~Villani,
\emph{Optimal Transport: Old and New},
Grundlehren der mathematischen Wissenschaften, vol.~338,
Springer-Verlag, Berlin, 2009,
doi:10.1007/978-3-540-71050-9.

\end{thebibliography}
\end{document}